\documentclass[12pt]{article} 

\usepackage[utf8]{inputenc} 
\usepackage{geometry} 
\usepackage{graphicx} 
\usepackage[bb=boondox,bbscaled=.95,cal=boondoxo]{mathalfa}
\usepackage[all,cmtip]{xy}
\usepackage{booktabs} 
\usepackage{array} 
\usepackage{paralist} 
\usepackage{verbatim} 
\usepackage{subfig} 

\usepackage{amsfonts,amsmath,latexsym,amssymb} 
\usepackage{amsthm}            
\usepackage{mathptmx}
\usepackage{fancyhdr} 
\usepackage[all,cmtip]{xy}
\usepackage{tikz-cd}
\usepackage{sectsty}
\allsectionsfont{\ttfamily\mdseries\upshape} 
\usepackage{titlesec}
\titleformat{\section}
  {\normalfont\sffamily\Large\bfseries\centering}
  {\thesection}{1em}{}
\titleformat{\subsection}
  {\normalfont\sffamily\large\bfseries\centering}
  {\thesubsection}{1em}{}

\usepackage[russian,english]{babel}
\usepackage[nottoc,notlof,notlot]{tocbibind} 
\usepackage[titles,subfigure]{tocloft} 

\newtheorem{thm}{\textsc{Theorem}}
\newtheorem*{thm*}{\textsc{Theorem}}
\newtheorem{cor}[thm]{\textsc{Corollary}}
\newtheorem{lem}[thm]{\textsc{Lemma}}
\newtheorem{prop}[thm]{\textsc{Proposition}}
\newtheorem*{prop*}{Proposition}
\newtheorem*{lem*}{\textsc{Lemma}}

\newtheorem*{rem*}{\textsc{Remark}}

\newtheorem*{exa*}{Example}
\theoremstyle{definition}
\newtheorem*{defn}{\textsc{Definition}}

\theoremstyle{remark}
\newtheorem*{remark}{\textsc{Remark}}

\DeclareMathOperator{\Hh}{\mathrm{H}}
\DeclareMathOperator{\Bb}{\mathrm{B}}
\DeclareMathOperator{\Aa}{\mathrm{A}}

\DeclareMathOperator{\id}{Id}

\DeclareMathOperator{\K}{\mathcal{K}}

\newcommand{\fact}{Lemma $2$}
\newcommand{\ev}{\mathrm{ev}}
\newcommand{\sumrep}[1]{\underset{v\lhd^{#1} w}{\text{\Large{$\oplus$}}}}
\newcommand{\sumr}{\text{\Large{$\oplus$}}}
\newcommand{\elll}{\mathcal{l}}
\newcommand{\G}{\mathrm{G}}
\newcommand{\ddef}{\mathrm{def}}
\newcommand{\im}{\mathrm{Im}}\usepackage[bb=boondox,bbscaled=.95,cal=boondoxo]{mathalfa}
\begin{document}

\title{$ q$-Independence of the Jimbo-Drinfeld Quantization
}


\author{Olof Giselsson}
\date{}
\maketitle

\begin{abstract}
Let $\mathrm G$ be a connected semi-simple compact Lie group and for $0<q<1$, let $(\mathbb{C}[\mathrm{G]_q}],\Delta_q)$ be the Jimbo-Drinfeld $q$-deformation of $\mathrm G$.  We show that the $C^*$-completions of $\mathrm{C}[\mathrm{G]_q}$ are isomorphic as $C^{*}$-algebras for all values of $q$. 
Moreover, these isomorphisms are equivariant with respect to the right-actions of the maximal torus.
\end{abstract}

\section{\bf\textsc{Introduction}}
The quantized universal enveloping algebra $\mathrm{U_{q}}(\mathfrak{g})$ of a semi-simple Lie algebra $\mathfrak g$ was introduced by Drinfeld and Jimbo in the mid-80's~\cite{dr,jimbo}.                     
In~\cite{dr2}, Drinfeld also introduced their dual objects, deformations $\mathbb{C}[\mathrm{G]_{q}}$ of the Hopf algebra of regular functions on a semi-simple Lie group $\mathrm G.$ Moreover, when $\mathrm G$ is compact, the algebra $\mathbb{C}[\mathrm{G]_{q}}$ can be given the structure of a Hopf $*$-algebra. In this case one can see that the enveloping $C^{*}$-algebra of $\mathbb{C}[\mathrm{G]_{q}}$ exists, giving a natural $q$-analogue $C(\mathrm{G)_{q}}$ of the algebra of continuous functions on $\mathrm{G}.$ The analytic approach to quantum groups was initially proposed by Woronowicz~\cite{wz1}. In the 90's Soibelman gave a complete classification of the irreducible $*$-representations of $\mathbb{C}[\mathrm{G]_{q}}$. These were shown to be in one-to-one correspondence with the symplectic leaves of $\mathrm{G}$ coming from the Poisson structure on $C(\mathrm G)$ determined by the quantization when $q\to 1$. However, it was not clear how the $C^{*}$-algebraic structure of $C(\mathrm{G)_{q}}$ was depending on the parameter $q.$ In fact, several evidence pointed towards that the structure was actually independent of it.
In the special case of $\mathrm{SU}_{2},$ it was observed (see~\cite{wz1}) that the $C^{*}$-algebras $C(\mathrm{SU_{2}})_{q},$ $q\in (0,1)$ are all isomorphic. In the mid 90's, G. Nagy showed in~\cite{gnagy} that the same holds for $C(\mathrm{SU_{3})_{q}}$. Moreover, it was also shown by Nagy (in~\cite{gnagy1}) that $C(\mathrm{SU_{n})_{q}}$ is KK-equivalent to $C(\mathrm{SU_{n})_{s}},$ for all $n\in \mathbb{N}$ and all $q,s\in (0,1).$ This was extended by Neshveyev-Tuset in~\cite{nt} to yield a KK-equivalence between $C(\mathrm{G)_{q}}$ and $C(\mathrm{G)_{s}}$ for any compact simply connected semi-simple Lie group G. In this paper, we show that some of the ideas that underpin Nagy's proof of the $q$-independence of $C(\mathrm{SU_{3})_{q}}$ can be extended to give the following result: for a fixed symplectic leaf $U\subseteq \mathrm{G},$ with corresponding $*$-representations $\pi^{q}$ of $\mathbb{C}[\mathrm{G]_{q}},$ we have an isomorphism $\overline{\im \pi^{q}}\cong \overline{\im\pi^{s}}$ for all $q,s\in (0,1).$ Using this, we prove that $C(\mathrm{G)_{q}}\cong C(\mathrm{G)_{s}},$ thus showing that these non-isomorphic compact quantum groups are all isomorphic as $C^{*}$-algebras. This confirms the conjecture made in~\cite{gnagy}.
\\
   
The paper is organized as follows. We finish this section by giving some geometric intuitions underlying the proof and describing the idea of the proof using this geometric picture. We then present and prove the lifting theorem used in the proof of the main result. Section $2$ goes through the formal definitions of $\mathbb{C}[\mathrm{G]_{q}}$ as well as its representation theory. In Section $3,$ we prove some more specific results regarding representations and how these depend on the $q$ parameter. In Section $4,$ we state and prove the main result.

\subsection{\bf\textsc{Outline of the Proof}}
  
To explain the main ideas of the proof, it is worth to start by considering the case of $\mathrm{G}=\mathrm{SU_{3}},$ previously covered by Nagy. There is an irreducible $*$-representation $\pi^{q}:\mathbb{C}[\mathrm{SU_{3}]_{q}}\rightarrow \mathcal{B}(\ell^{2}(\mathbb{Z}_{+})^{\otimes 3})$ that, in Nagy's words, corresponds to the "big symplectic leaf" in $\mathrm{SU_{3}}.$ However, there is an inherent problem when trying to determine if $\overline{\im \pi^{q}}\cong \overline{\im \pi^{s}}$ for different $q,s\in (0,1).$ As $\pi^{q}$ can be seen to vary (in a certain sense) continuously on $q,$ one intuitive approach could be to let $q\to 0,$ and then find some natural set of generators such that an isomorphic set of generators can be found in $\overline{\im \pi^{q}}$ for each $q\in (0,1).$ This method works, for example, in the case of $C(\mathrm{SU_{2})_{q}}.$ However, for $\mathbb{C}[\mathrm{SU_{3}]_{q}},$ there seems to be no simple way of taking the limit $q\to 0.$ As it sits, the image of $\pi^{q}$ is simply too "twisted" to allow passing to any limit. 
\\

These problems are resolved in the following way: We observe that $\overline{\im \pi^{q}},$ for all $q\in (0,1),$ contains the compact operators $\K\subseteq \mathcal{B}(\ell^{2}(\mathbb{Z}_{+})^{\otimes 3}),$ and moreover, the intersection $\im \pi^{q}\cap\K$ is non-trivial. The compact operators form a minimal ideal in $\overline{\im \pi^{q}},$ in the sense that it is contained in any other ideal. We now consider the composition
$$
\mathbb{C}[\mathrm{SU_{3}]_{q}}\overset{\pi^{q}}\longrightarrow \mathcal{B}(\ell^{2}(\mathbb{Z}_{+})^{\otimes 3})\overset{p}\longrightarrow \mathcal{B}(\ell^{2}(\mathbb{Z}_{+})^{\otimes 3})/\K\cong \mathcal{Q}(\ell^{2}(\mathbb{Z}_{+})^{\otimes 3}),
$$
where $p$ is the quotient map $x\mapsto x+\K,$ and we then proceed by analyzing $p\circ \pi^{q}.$ It is clear that $\pi^{q}$ can not be a direct summand in $p\circ \pi^{q},$ as the elements mapped by $\pi^{q}$ into $\im \pi^{q}\cap\K$ is now mapped into zero. It turns out that there are two Hilbert spaces $\Hh_{1},$ $\Hh_{2},$ such that for every $q\in (0,1),$ we have $*$-representations 
$$\begin{array}{ccc}\mathbb{C}[\mathrm{SU_{3}]_{q}}\overset{\Pi_{1}^{q}}\longrightarrow \mathcal{B}(\Hh_{1}), &\mathbb{C}[\mathrm{SU_{3}]_{q}}\overset{\Pi_{2}^{q}}\longrightarrow \mathcal{B}(\Hh_{2})\end{array}$$
and an isomorphism $\varphi_{q}:\overline{\im \pi^{q}}/\K\to \overline{\im (\Pi_{1}^{q}\oplus \Pi_{2}^{q})},$ such that for every $a\in \mathbb{C}[\mathrm{SU_{3}]_{q}},$ we have $\varphi_{q}(p\circ\pi^{q}(a))=(\Pi_{1}^{q}\oplus\Pi_{2}^{q})(a).$ Moreover, in this case, one can show that actually we have $\overline{\im (\Pi_{1}^{q}\oplus \Pi_{2}^{q})}=\overline{\im (\Pi_{1}^{s}\oplus \Pi_{2}^{s})}$ for all $q,s\in (0,1)$ as subspaces of $\mathcal{B}(\Hh_{1})\oplus \mathcal{B}(\Hh_{2}).$ Thus, by quoting out the compact operators, we have successfully "untwisted" the $*$-representation $\pi^{q}.$ Letting $M=\overline{\im (\Pi_{1}^{q}\oplus \Pi_{2}^{q})},$ one then shows that the injective homomorphisms $\varphi_{q}^{-1}:M\to \mathcal{Q}(\ell^{2}(\mathbb{Z}_{+})^{\otimes 3})$ varies norm-continuously on $q,$ meaning that, for a fixed $x\in M,$ the map $q\in (0,1)\to \varphi_{q}^{-1}(x)$ is a continuous function of $q.$ We then get an isomorphism \begin{equation}\label{first}\begin{array}{ccc}p^{-1}(\varphi^{-1}_{q}(M))\cong p^{-1}(\varphi^{-1}_{s}(M)), & q,s\in (0,1)\end{array}\end{equation}
using the lifting result (Lemma~\ref{gnagy} below). As $\K\subseteq \overline{\im \pi^{q}},$ it follows that $p^{-1}(\varphi^{-1}_{q}(M))=\overline{\im \pi^{q}}$ and hence we have established an isomorphism for different $q.$ One needs further argument to conclude that actually $C(\mathrm{SU_{3})_{q}}\cong C(\mathrm{SU_{3})_{s}},$ but by proving~\eqref{first}, the main effort is done. The proof below essentially systematizes this line of argument in a way that makes it also work for a general $\mathrm G$, by quoting out ideals to "untwist" an irreducible $*$-representation $\pi^{q}$ of $C(\mathrm{G})_{q}$ further and further, until it is clear that the images of the resulting $*$-representations are independent of $q.$ Then one uses inductive arguments to check that also $\overline{\im\pi^{q}}\cong \overline{\im\pi^{s}}$ for $q,s\in (0,1).$
\\

One can give a quite clear geometric heuristic of this, using the one-to-one correspondence between irreducible $*$-representations of $\mathbb{C}[\mathrm{G]_{q}}$ and symplectic leaves in $\mathrm G$ coming from the corresponding Poisson structure on $C(\mathrm G)$ (see~\cite{ks}). Recall that $\mathrm G$ can be decomposed into a disjoint union of symplectic leaves and that each leaf is an even-dimensional sub-manifold of $\mathrm G$. Let $U$ be a 2m-dimensional symplectic leaf of $\mathrm G,$ corresponding to a $*$-representations $\pi^{q}$ of $\mathbb{C}[\mathrm{G]_{q}},$ and let $C_{0}(U)$ be the ideal of $C(\overline{U})$ of all continuous functions vanishing on $\overline{U}\backslash U.$ Thus quoting out this ideal gives a homomorphism $C(\overline{U})\to C(\overline{U}\backslash U).$ It turns out that $\overline{U}\backslash U$ can be written as a disjoint union $\cup_{j} U_{j},$ of symplectic leaves of dimension strictly less than $2m,$ and that the leaves in this union of dimension $<2m-2$ are contained in the closures of the leaves of dimension $2m-2.$ In general, we can write $\overline{U}$ as a disjoint union of symplectic leaves
 $$\overline{U}=U\cup \left(\cup_{j} U_{j}^{(m-1)}\right)\cup \left(\cup_{j} U_{j}^{(m-2)}\right)\cup\cdots \cup \left(\cup_{j} U_{j}^{(0)}\right)$$ such that each $U_{j}^{(k)}$ is a symplectic leaf of dimension $2k$ and 

\begin{equation}\label{symplec}
\begin{array}{ccc}
\cup_{j} \overline{U}_{j}^{(k)}=\left(\cup_{j} U_{j}^{(k)}\right)\cup\cdots \cup \left(\cup_{j} U_{j}^{(0)}\right).
\end{array}
\end{equation}
This shows that we can make a sequence of homomorphisms
\begin{equation}\label{keke}
C(\overline{U})\longrightarrow \prod_{j} C(\overline{U}_{j}^{(m-1)})\longrightarrow\dots \longrightarrow \prod_{j} C(\overline{U}_{j}^{(1)})\longrightarrow \prod_{j} C(\overline{U}_{j}^{(0)})
\end{equation}
such that on each step, the homomorphism $\prod_{j} C(\overline{U}_{j}^{(k)})\to \prod_{j} C(\overline{U}_{j}^{(k-1)})$ has kernel $\prod_{j} C_{0}(U_{j}^{(k)}).$
Let us explain how a $q$-analogue of~\eqref{keke} is used. \begin{remark}For several reasons, the notations used here will differ somewhat from the ones used later in the text.\end{remark} Let $U$ and $U^{(k)}_{j}$ be as above. We can think of $\overline{\im\pi^{q}}$ and the ideal $\K\subseteq \overline{\im\pi^{q}}$ as $q$-analogs of $C(\overline{U})$ and $C_{0}(U),$ denoted by $C(\overline{U})_{q}$ and $C_{0}(U)_{q}$ respectively. There is then a sequence of homomorphisms
\begin{equation}\label{keke1}
C(\overline{U})_{q}\overset{\partial_{m}^{q}}\longrightarrow \prod_{j} C(\overline{U}_{j}^{(m-1)})_{q}\overset{\partial_{m-1}^{q}}\longrightarrow\dots \overset{\partial_{2}^{q}}\longrightarrow \prod_{j} C(\overline{U}_{j}^{(1)})_{q}\overset{\partial_{1}^{q}}\longrightarrow \prod_{j} C(\overline{U}_{j}^{(0)})_{q}
\end{equation}
such that on each step, the homomorphism $\partial_{k}:\prod_{j} C(\overline{U}_{j}^{(k)})_{q}\to \prod_{j} C(\overline{U}_{j}^{(k-1)})_{q}$ has kernel equal to $\prod_{j} C_{0}(U_{j}^{(k)})_{q}.$ Let us denote by $C(\partial^{(k)}\overline{U})_{q},$ $k=0,\dots,m-1,$ the image of $C(\overline{U})_{q}$ in $\prod_{j} C(\overline{U}_{j}^{(k)})_{q}$ via the composition of homomorphisms in~\eqref{keke1}. The idea is to proceed by induction on the dimensions of the symplectic leaves. In the case of zero-dimensional leaves, the corresponding $*$-representations are one-dimensional (maps to $\mathbb{C}$) and hence trivially $q$-independent. For higher dimensional leaves, we can use the induction hypothesis to connect the lower dimensional leaves for different $q,s\in (0,1)$
\begin{equation}\label{keke2}
\begin{xy}\xymatrixrowsep{3pc}\xymatrix{
C(\overline{U})_{q}\ar@{-->}[d] \ar[r]^-{\partial_{m}^{q}} &  \prod_{j} C(\overline{U}_{j}^{(m-1)})_{q} \ar[d]_{\Gamma^{s,q}_{m-1}}  \ar[r]^-{\partial_{m-1}^{q}}  & \cdots  \ar[r]^-{\partial_{2}^{q}} &\prod_{j} C(\overline{U}_{j}^{(1)})_{q}\ar[r]^-{\partial_{1}^{q}}\ar[d]_{\Gamma^{s,q}_{1}}  &\prod_{j} C(\overline{U}_{j}^{(0)})_{q}\ar[d]_{\Gamma^{s,q}_{0}} \\
C(\overline{U})_{s} \ar[r]^-{\partial_{m}^{s}} &  \prod_{j} C(\overline{U}_{j}^{(m-1)})_{s}  \ar[r]^-{\partial_{m-1}^{s}}  & \cdots  \ar[r]^-{\partial_{2}^{s}} &\prod_{j} C(\overline{U}_{j}^{(1)})_{s}\ar[r]^-{\partial_{1}^{s}} &\prod_{j} C(\overline{U}_{j}^{(0)})_{s}\\
}\end{xy}
\end{equation}

Moreover, this can be done in a way such that the diagram~\eqref{keke2} is commutative. The aim is then to construct a dotted arrow from $C(\overline{U})_{q}$ to $C(\overline{U})_{s}$ that makes the diagram commutative. The main obstacle to do this is to check that $C(\partial^{(m-1)}\overline{U})_{q}$ is mapped by $\Gamma^{s,q}_{m-1}$ to $C(\partial^{(m-1)}\overline{U})_{s}.$ In order to prove this, one shows the following
\begin{enumerate}[(i)]
\item for $k=0,\dots, m-1$ the intersection of $C(\partial^{(k)}\overline{U})_{q}$ with $\prod C_{0}(U_{j}^{(k)})_{q}$ is mapped by $\Gamma^{s,q}_{k}$ to the intersection of $C(\partial^{(k)}\overline{U})_{s}$ with $\prod C_{0}(U_{j}^{(k)})_{s},$ 
\item
the $C^{*}$-algebras $C(\partial^{(0)}\overline{U})_{q}$ are commutative and isomorphic for all $q\in (0,1),$ via $\Gamma^{s,q}_{0}$,
\item 
for $k=1,\dots, m-1,$ there is an approximate unit $\{u_{q,i}^{(k)}\}_{i=1}^{\infty}$ for $\prod_{j} C_{0}(U_{j}^{(k)})_{q}$ such that $\{u_{q,i}^{(k)}\}_{i=1}^{\infty}\subseteq C(\partial^{(k)}\overline{U})_{q}$ and $\Gamma_{k}^{s,q}(u_{q,i}^{(k)})=u_{s,i}^{(k)}$ for all $i\in \mathbb{N}.$
 \end{enumerate}
\begin{remark}In the actual proof below, we do not really use $(iii),$ since by the way the arguments are constructed there, explicitly stating this point becomes unnecessary.\end{remark}
From the commutivity of the square

\begin{equation}\label{comsq}
\begin{xy}\xymatrix{
\prod_{j} C(\overline{U}_{j}^{(k)})_{q}\ar[r]^-{\partial_{k}^{q}}\ar[d]_{\Gamma^{s,q}_{k}} & \prod_{j} C(\overline{U}_{j}^{(k-1)})_{q}\ar[d]_{\Gamma^{s,q}_{k-1}}\\
\prod_{j} C(\overline{U}_{j}^{(k)})_{s} \ar[r]^-{\partial_{k}^{s}} &\prod_{j} C(\overline{U}_{j}^{(k-1)})_{s}
}\end{xy}
\end{equation}
it follows that if $\Gamma^{s,q}_{k-1}$ restricts to a $*$-isomorphism from $C(\partial^{(k-1)}\overline{U})_{q}$ to $ C(\partial^{(k-1)}\overline{U})_{s},$ then as for any $x\in C(\partial^{(k)}\overline{U})_{q},$ we have $$\partial_{k}^{s}(\Gamma^{s,q}_{k}(x))=\Gamma^{s,q}_{k-1}(\partial_{k}^{q}(x))\in C(\partial^{(k-1)}\overline{U})_{s},$$ it follows that $\Gamma^{s,q}_{k}(x)=y+c$ where $y\in C(\partial^{(k)}\overline{U})_{s}$ and $c\in \prod_{j} C_{0}(U_{j}^{(k)})_{q}.$ By $(iii),$ we have an approximate unit $\{u_{q,i}^{(k)}\}_{i=1}^{\infty}$ such that $x u_{q,i}^{(k)}$ is in the intersection of $C(\partial^{(k)}\overline{U})_{q}$ with $\prod_{j} C_{0}(U_{j}^{(k)})_{q}.$ It now follows from $(i)$ that for all $i\in \mathbb{N}$, we have $$\Gamma_{k}^{s,q}(xu_{q,i}^{(k)})\in C(\partial^{(k)}\overline{U})_{s},$$ 
$$y\Gamma_{k}^{s,q}(u_{q,i}^{(k)})\in C(\partial^{(k)}\overline{U})_{s}$$
and thus it follows that also $c u_{s,i}^{(k)}\in C(\partial^{(k)}\overline{U})_{s}.$ Letting $i\to \infty$ now gives $c\in C(\partial^{(k)}\overline{U})_{s}$ and thus $\Gamma_{k}^{s,q}(x)\in C(\partial^{(k)}\overline{U})_{s}.$ Using $(ii),$ it now follows by induction that $C(\partial^{(m-1)}\overline{U})_{q}$
is mapped isomorphically onto $C(\partial^{(m-1)}\overline{U})_{s}.$ As $\ker \partial^{q}_{m}=C_{0}(U)_{q},$ this is equivalent to $$\begin{array}{cccc}C(\overline{U})_{q}/C_{0}(U)_{q}\cong C(\overline{U})_{s}/C_{0}(U)_{s}, & q,s\in (0,1).\end{array}$$ After checking that these isomorphisms are varying norm-continuously as functions of $q$ and $s$, the dotted arrow in~\eqref{keke2} can then be constructed using \fact~below.
\subsection{\bf\textsc{Lifting Results}}

\begin{lem}[{Lemma $2$ in~\cite{gnagy}}]\label{gnagy}
Let $\Hh$ be a separable Hilbert space, let $\K$ be the space of compact operators on $\Hh,$ let $\mathcal{Q}(\Hh)=\mathcal{B}(\Hh)/\K$ be the Calkin algebra and $p: \mathcal{B}(\Hh)\to \mathcal{Q}(\Hh)$ the quotient map. Suppose $A$ is a fixed separable $C^{*}$-algebra of type $\mathrm{I}$ and $\phi_{q}:A\to Q(\Hh),$ $q\in [0,1]$ is a point-norm continuous family of injective $*$-homomorphisms. Denote  $$\mathfrak{A}_{q}:=\phi_{q}(A):$$ 
$$
M_{q}:=p^{-1}(\mathfrak{A}_{q}).
$$
Then there exists a family of injective $*$-homomorphisms $\Phi_{q}:M_{0}\to \mathcal{B}(\Hh),$ $q\in [0,1]$ with the following properties
\begin{enumerate}[(a)]
\item $\Phi_{q}(M_{0})=M_{q}$ for $q\in [0,1]$ and $\Phi_{0}=\id_{M_{0}},$
\item the family $\Phi_{q}:M_{0}\to \mathcal{B}(H),$ $q\in[0,1]$ is point-norm continuous,
\item for every $q\in [0,1],$ the diagram
\begin{equation}\label{com1}
\begin{xy}\xymatrix{
M_{0}\ar[r]^*{\Phi_{q}}\ar[d]_*{p}& M_{q}\ar[d]^*{p}\\
\mathfrak{A}_{0} \ar[r]_*{\phi_{q}\circ \phi_{0}^{-1}} & \mathfrak{A}_{q}
}\end{xy}
\end{equation}
is commutative.
\end{enumerate}
\end{lem}
We remind the reader that a $C^{*}$-algebra of type $\mathrm I,$ is one where the image of every irreducible $*$-representation includes a non-zero compact operator. Here, we will use a modified version of Lemma~\ref{gnagy}.

\begin{lem}\label{1}
Let $\Hh$ be a separable Hilbert space, let $\K$ be the space of compact operators on $\Hh,$ let $\mathcal{Q}(\Hh)=\mathcal{B}(\Hh)/\K$ be the Calkin algebra and $p: \mathcal{B}(\Hh)\to \mathcal{Q}(\Hh)$ be the quotient map. For every $q\in (0,1),$ suppose $\Aa_{q}\subseteq \mathcal{Q}(\Hh)$ is a separable $C^{*}$-algebra of type $\mathrm{I}$ and we have a family of $*$-isomorphisms $\phi_{s,q}:\Aa_{q}\to \Aa_{s},$ $s,q\in (0,1)$ which are continuous in the point-norm topology (i.e. for every fixed $q\in (0,1)$ and $x\in \Aa_{q},$ the map $s\in (0,1)\mapsto \phi_{s,q}(x)\in \mathcal{Q}(\Hh)$ is norm-continuous). Assume moreover that $\phi_{q,q}=\id_{\Aa_{q}}$ and $\phi_{t,s}\circ\phi_{s,q}=\phi_{t,q}$ for all $t,s,q\in (0,1).$ Denote
$$
\Bb_{q}:=p^{-1}(\Aa_{q}).
$$
Then there exists a family of inner $*$-isomorphisms $\Phi_{s,q}:\mathcal{B}(\Hh)\to\mathcal{B}(\Hh),$ $s,q\in(0,1)$ (i.e $\Phi_{s,q}(x)=U_{s,q}^{*}xU_{s,q}$ for some unitary $U_{s,q}\in \mathcal{B}(\mathrm{H})$) with the following properties
\begin{enumerate}[(i)]
\item $\Phi_{s,q}(\Bb_{q})=\Bb_{s}$ for $s,q\in (0,1),$
\item $\Phi_{q,q}=\id$ and $\Phi_{t,s}\circ \Phi_{s,q}=\Phi_{t,s}$ for all $t,s,q\in (0,1),$
\item for fix $q,$ the family $\Phi_{s,q}:\Bb_{s}\to \mathcal{B}(H),$ $s\in[0,1]$ is point-norm continuous,
\item for all $s,q\in (0,1),$ the diagram
\begin{equation}\label{com}
\begin{xy}\xymatrix{
\Bb_{q}\ar[r]^*{\Phi_{s,q}}\ar[d]_*{p}& \Bb_{s}\ar[d]^*{p}\\
\Aa_{q} \ar[r]_*{\phi_{s,q}} & \Aa_{s}
}\end{xy}
\end{equation}
is commutative.
\end{enumerate}
\end{lem}
\begin{proof}
If $a<b,$ then clearly the conclusion of Lemma~\ref{gnagy} still holds if we change the interval to $[a,b].$ Let $a_{k}\in (0,1),$ $k\in \mathbb{Z}$ be a strictly increasing sequence such that $a_{k}\to 1$ and $a_{-k}\to 0$ as $k\to \infty.$ For $k\geq 0,$ we apply Lemma~\ref{gnagy} to the set of injective $*$-homomorphisms $$\begin{array}{ccc}\tilde\phi_{q,k}:=\phi_{q,a_{k}}:\Aa_{a_{k}}\to \mathcal{Q}(\Hh),& q\in [a_{k},a_{k+1}],\end{array}$$
and let $M_{q,k}=p^{-1}(\tilde\phi_{q,k}(\Aa_{a_{k}}))=p^{-1}(A_{q}).$
For $k<0,$ we instead apply Lemma~\ref{gnagy} to 
$$\begin{array}{ccc}\tilde\phi_{q,k}:=\phi_{q,a_{k+1}}:\Aa_{a_{k+1}}\to \mathcal{Q}(\Hh),& q\in [a_{k},a_{k+1}].\end{array}$$
Let $\tilde\Phi_{q,k},$ $k\in \mathbb{Z},$ be the $*$-isomorphisms aquired by applying Lemma~\ref{gnagy}. Note that $\tilde\Phi_{q,k}$ is an isomorphism from $\Bb_{a_{k}}$ to $ \Bb_{q}$ for $q\in [a_{k},a_{k+1}]$ and $k\geq 0$ and an isomorphism from $\Bb_{a_{k+1}}$ to $\Bb_{q}$ if $k<0.$ Let us define a $*$-isomorphism $\Phi_{q}:\Bb_{a_{0}}\to \Bb_{q}$ by the formula
\begin{equation}\label{phimap}
\Phi_{q}=
\begin{cases}
\tilde\Phi_{q,k}\circ\tilde \Phi_{a_{k},k-1}\circ \cdots \circ\tilde \Phi_{a_{1},0}, & \text{if $q\in [a_{k},a_{k+1}]$ and $k\geq 0$}\\
\tilde\Phi_{q,k}\circ \tilde\Phi_{a_{k},k+1}\circ \cdots \circ\tilde\Phi_{a_{-1},-1}, & \text{if $q\in [a_{k},a_{k+1}]$ and $k< 0$}
\end{cases}
\end{equation}
 It follows from Lemma~\ref{gnagy} and the construction of $\Phi_{q}$ that for $x\in \Bb_{a_{0}},$ the map $q\in (0,1)\mapsto \Phi_{q}(x)\in \mathcal{B}(H)$ is norm-continuous. That $p\circ \Phi_{q}=\phi_{q,a_{0}}\circ p$ holds follows from $\phi_{t,s}\circ\phi_{s,q}=\phi_{t,q}$ and iteration of the commutative diagram
$$
\begin{xy}\xymatrix{
\Bb_{a_{0}}\ar[r]^*{\tilde\Phi_{a_{1},0}}\ar[d]_*{p} &\Bb_{a_{1}}\ar[r]^*{\tilde\Phi_{q,1}}\ar[d]_*{p}& \Bb_{q}\ar[d]^*{p} \\
\Aa_{a_{0}}\ar[r]_*{\phi_{a_{1},a_{0}}}&\Aa_{a_{1}} \ar[r]_*{\phi_{q,a_{1}}} & \Aa_{q} 
}\end{xy}
$$
 We then let $\Phi_{s,q}=\Phi_{s}\circ \Phi_{q}^{-1}$ for $s,q\in (0,1).$ By~\eqref{com1}, we have a commutative diagram 
$$
\begin{xy}\xymatrix{
\Bb_{q}\ar[r]^*{\Phi_{q}^{-1}}\ar[d]_*{p} &\Bb_{a_{0}}\ar[r]^*{\Phi_{s}}\ar[d]_*{p}& \Bb_{s}\ar[d]^*{p} \\
\Aa_{q}\ar[r]_*{\phi_{q,a_{0}}^{-1}}&\Aa_{a_{0}} \ar[r]_*{\phi_{s,a_{0}}} & \Aa_{s} 
}\end{xy}
$$

From this, we get the commutivity of~\eqref{com}, as $$\phi_{s,a_{0}}\circ \phi_{q,a_{0}}^{-1}=\phi_{s,a_{0}}\circ\phi_{a_{0},q}=\phi_{s,q}.$$
Hence, $(i),$ $(ii),$ $(iii),$ and $(iv)$ holds for $\Phi_{s,q}$ as a $*$-isomorphism $\Bb_{q}\to \Bb_{s}.$ 
\\

We now extend $\Phi_{s,q}$ as an inner automorphism of all of $\mathcal{B}(\Hh)$. To do this, we first show that the restriction $\Phi_{s,q}|_{\K}$ is inner.
We have $\K\subseteq \Bb_{q},$ and we get from the diagram~\eqref{com} that $\Phi_{s,q}(\K)\subseteq\K,$ and as $\Phi_{s,q}^{-1}=\Phi_{q,s}$ it follows that actually \begin{equation}\label{k}\Phi_{s,q}(\K)=\K.\end{equation} 
So $\Phi_{s,q}$ is an irreducible representation of $\K.$ It is known that any such is unitarily equivalent to the identity representation (e.g. see Corollary~$1.10$ in~\cite{dav}). Hence, there exists a unitary $U_{s,q}\in \mathcal{B}(\Hh)$ such that $\Phi_{s,q}(x)=U_{s,q}x U_{s,q}^{*}$ for all $x\in \K.$ For arbitrary $y\in \Bb_{q},$ we obtain $$\Phi_{s,q}(x)\Phi_{s,q}(y)=\Phi_{s,q}(xy)=U_{s,q}x yU_{s,q}^{*}=\Phi_{s,q}(x)U_{s,q}yU_{s,q}^{*}$$ for all $x\in \K.$ This gives $\Phi_{s,q}(y)=U_{s,q}y U_{s,q}^{*}.$
\end{proof}
\section{\bf\textsc{Preliminaries}}
In this section we recall some facs about the Hopf algebras $\mathrm{U_{q}}(\mathfrak{g})$ and $\mathbb{C}[\mathrm{G]_{q}}.$ The presentation is mainly taken from~\cite{nt}. A general reference for the technical claims made here is~\cite{ks}.
\subsection{\bf\textsc{The Quantum Group $\mathrm{U_{q}}(\mathfrak{g})$}}
Let $\mathrm G$ be a simply connected semisimple compact Lie group and let $\mathfrak{g}$ denote its complexified Lie algebra. We write $\mathrm U (\mathfrak{g})$ to denote the universal enveloping algebra of $\mathfrak{g}$ equipped with a $*$-involution induced by the real form derived from $\G.$ Moreover, we let $\mathfrak{h}\subseteq \mathfrak{g}$ be the Cartan sub-algebra coming from a maximal torus $\mathrm T\subseteq \G.$ Let $\mathfrak{t}\subseteq \mathfrak{h}$ be the real subspace of skew-symmetric elements. Write $\Phi$ for the set of roots of $\mathfrak{g}$, $\Phi_{+}$ for the set of positive roots and $\Omega=\{\alpha_{1},\dots,\alpha_{n}\}$ for the set of simple roots.
\\

We denote the Weyl group of $\mathfrak g$ by $W$ and identify its set of generators $s_{i},$ $i=1,\dots, n,$ (as a Coxeter group) with $\Omega$ by the identification $\alpha_{i}\mapsto s_{\alpha_{i}}=:s_{i}.$ Moreover, we identify $\Phi_{+}$ with the set $\{w s_{i} w^{-1}:s_{i}\in \Omega,w\in W\}\subseteq W.$ In both instances, we write the identification as $\alpha\mapsto h_{\alpha}\in \mathfrak{t}.$
\\

Let $q\in (0,1).$ Let $(a_{ij})_{ij}$ be the Cartan matrix of $\mathfrak{g}$ and $d_{i}=\frac{(\alpha_{i},\alpha_{i})}{2}$ for $i=1,\dots,n$. Let $q_{i}=q^{d_{i}}$ for $i=1,\dots, n.$ The quantized universal enveloping algebra $\mathrm{U_{q}}(\mathfrak{g})$ is the unital complex algebra generated by elements $E_{i},F_{i},K_{i},K_{i}^{-1},$ $i=1,\dots, n$ subject to the relations
$$
\begin{array}{cccc}
K_{i}K^{-1}_{i}=K^{-1}_{i}K_{i}=1, & K_{i}K_{j}=K_{j}K_{i}, & K_{i}E_{j}=q_{i}^{a_{ij}}E_{j}K_{i}, & K_{i}F_{j}=q_{i}^{a_{ij}}F_{i}K_{j},
\end{array}
$$
$$
E_{i}F_{j}-F_{j}E_{i}=\delta_{ij}\frac{K_{i}-K_{i}^{-1}}{q_{i}-q_{i}^{-1}}
$$
$$
\sum_{k=0}^{1-a_{ij}}(-1)^{k}\left[ \begin{array}{cc}1-a_{ij}\\ k  \end{array}\right]_{q_{i}}E_{i}^{k}E_{j}E_{i}^{1-a_{ij}-k}=0, 
$$
$$
 \sum_{k=0}^{1-a_{ij}}(-1)^{k}\left[ \begin{array}{cc}1-a_{ij}\\ k  \end{array}\right]_{q_{i}}F_{i}^{k}F_{j}F_{i}^{1-a_{ij}-k}=0
$$
where $\left[ \begin{array}{cc}k\\ j  \end{array}\right]_{q_{i}}=\overset{j-1}{\underset{m=0}\prod}\frac{q_{i}^{-(k-m)}-q_{i}^{k-m}}{q_{i}^{-(m+1)}-q_{i}^{m+1}}$ is a $q$-analogue of the binomial coefficients. 
\\

$\mathrm{U_{q}}(\mathfrak{g})$ becomes a Hopf $*$-algebra when equipped with a co-associative co-product $\hat\Delta_{q},$ a co-unit $\hat\epsilon_{q},$ an antipode $\hat S_{q}$ and a $*$-involution given on generators as
$$
\begin{array}{cccc}
\hat\Delta_{q}(K_{i})=K_{i}\otimes K_{i}, & \hat\Delta_{q}(E_{i})=E_{i}\otimes 1+K_{i}\otimes E_{i}, &\hat\Delta_{q}(F_{i})=F_{i}\otimes K^{-1}_{i}+1\otimes F_{i},
\end{array}
$$ 
$$
\begin{array}{cccc}
\hat\epsilon_{q}(E_{i})=\hat\epsilon_{q}(F_{i})=0, & \hat\epsilon_{q}(K_{i})=1,
\end{array}
$$
$$
\begin{array}{cccc}
\hat S_{q}(F_{i})=-F_{i}K_{i}, & \hat S_{q}(E_{i})=-K^{-1}_{i}E_{i}, & \hat S_{q}(K_{i})=K_{i}^{-1}
\end{array}
$$
$$
\begin{array}{cccc}
K_{i}^{*}=K_{i}, & E_{i}^{*}=F_{i}K_{i}, & F_{i}^{*}=K_{i}^{-1}E_{i}.
\end{array}
$$
We denote the antipode of $\mathrm{U_{q}}(\mathfrak{g})$ by $\hat S_{q}.$ For $q=1,$ we let $\mathrm{U_{1}}(\mathfrak{g})=\mathrm{U}(\mathfrak{g})$ be the ordinary universal enveloping algebra with the usual Hopf $*$-algebra structure and denote the co-product, co-unit, antipode simply by $\hat\Delta,$ $\hat \epsilon,$ $\hat S.$
\\

Let $P$ be the set of weights for $\mathfrak{g}.$ Let $P_{+}\subseteq P$ be the set of dominant integral weights, i.e. $\lambda\in P$ such that $\langle \lambda,\alpha_{i} \rangle\geq 0$ for $i=1,\dots,n.$ Moreover, let $P_{++}\subseteq P_{+}\subseteq P$ be the set of those dominant weights $\lambda$ such that $\langle \lambda, \alpha_{i} \rangle>0$ for $i=1,\dots ,n.$
\\

The theory of $\mathrm{U_{q}}(\mathfrak{g})$-modules is very similar to the case $q=1$ (see~\cite{ks}). It is well known that for every $q\in(0,1),$ the monoidal category $\mathcal{M}_{q}(\mathfrak{g})$ of 
admissible finite dimensional $\mathrm{U_{q}}(\mathfrak{g})$-modules are
parameterized by $\lambda \in P_{+},$ with the same fusion rules as $\mathcal{M}(\mathfrak{g}),$
the monoidal categories of finite dimensional $\mathrm{U}(\mathfrak{g})$-modules. 
Moreover, for any $\lambda \in P_{+},$ the vector spaces $V_{\lambda}^{q}$ and $V_{\lambda}$ have the same dimension. There is a similar decomposition into weight sub-spaces $V_{\lambda}^{q}(\gamma)\subseteq V^{q}_{\lambda},$ for $\gamma\in P;$ these are the sub-spaces such that \begin{equation}\label{ki}\begin{array}{cccc}K_{i}\cdot\eta=q_{i}^{\gamma(H_{i})}\eta, & \eta\in V_{\lambda}^{q}(\gamma), & i=1,\dots, n. \end{array}\end{equation}
For each $\gamma\in P,$ we also have a vector space isomorphism
$$
V_{\lambda}^{q}(\gamma)\cong V_{\lambda}(\gamma).
$$
In particular, the sub-space $V_{\lambda}^{q}(\lambda)$ is one-dimensional and is the highest weight-space of $V_{\lambda}^{q}$, in the sense that for $\xi\in V_{\lambda}^{q}(\lambda),$ we have
$$
\begin{array}{ccc}
E_{i}\cdot\xi=0,& \text{for $i=1,\dots, n.$}
\end{array}
$$
There is a non-degenerate inner product $\langle\cdot,\cdot \rangle$ on $V_{\lambda}^{q}$ such that 
$$
\begin{array}{ccccc}
\langle a\cdot \eta,\xi\rangle=\langle \eta,a^{*}\cdot \xi\rangle, & \forall a\in \mathrm{U_{q}}(\mathfrak{g}),&  \forall\eta,\xi\in V_{\lambda}^{q}.\\
\end{array}
$$
Clearly, with respect to this inner product, different weight sub-spaces $V_{\lambda}^{q}(\gamma)$ are orthogonal for different $\gamma$. For any $w\in W,$ the Weyl group of $\mathrm G,$ the weight sub-space $V_{\lambda}^{q}(w\cdot \lambda)$ is one dimensional. Thus we can choose a unit vector $\xi_{w\cdot\lambda}\in V_{\lambda}^{q}(w\cdot \lambda).$ Let us do this in such way that $\xi_{w\cdot\lambda}=\xi_{v\cdot\lambda}$ if $w\cdot\lambda=v\cdot\lambda$ and thus no ambiguity arises from this notation. 
\\

For $w\in W,$ let $\elll(w)\in\mathbb{Z}_{+}$ denote the length of $w.$ By definition, this is the smallest integer $m$ such that $w$ can be written as a product of $m$ generators
\begin{equation}\label{reduced}
\begin{array}{ccccc}
w=s_{i_{1}}\cdots s_{i_{m}}.
\end{array}
\end{equation}
For the identity element $e\in W,$ we let $\elll(e)=0.$ If $\elll(w)=m$ in~\eqref{reduced} then the product $w=s_{i_{1}}\cdots s_{i_{m}}$ is said to be \textit{reduced}. Recall that the Bruhat order on $W$ is the partial ordering generated by declaring that
\begin{equation}\label{bruhato}
\begin{array}{ccc}
v<w ,& \text{if $\exists \alpha\in \Phi_{+}$ such that $v\alpha=w$ and $\elll(v)=\elll(w)-1.$}
\end{array}
\end{equation}
Let $\mathrm{U}_{q}(\mathfrak{b})$ be the sub-algebra of $\mathrm{U_{q}}(\mathfrak{g})$ generated by $K_{i},K_{i}^{-1},E_{i}$ for $i=1,\dots, n.$ We can concretely connect the Bruhat order of $W$ with certain $\mathrm{U_{q}}(\mathfrak{g})$-modules in the following way:
\begin{lem}[{Proposition~3.3 in~\cite{nt}}]\label{lem1}

Let $v,w\in W.$ The following are equivalent:
\begin{enumerate}[(i)]
\item We have $v\leq w$ in the Bruhat order on $W.$
\item For any $\lambda\in P_{++}$ we have $V_{\lambda}^{q}(v\cdot \lambda)\subset (\mathrm{U_{q}}(\mathfrak{b}))V_{\lambda}^{q}(w\cdot \lambda)$.
\end{enumerate}
\end{lem}
Both the subspaces $V_{\lambda}^{q}(v\cdot \lambda)$ and $(\mathrm{U_{q}}(\mathfrak{b}))V_{\lambda}^{q}(w\cdot \lambda)$ are invariant under the actions of $K_{i},$ $i=1,\dots, n.$ Notice that these are commuting self-adjoint operators on $V^{q}_{\lambda}$ that separate the weight-spaces. As the space $V_{\lambda}^{q}(v\cdot \lambda)$ is one-dimensional, it follows from~\eqref{ki} that if $v \not \leq w$ and hence $V_{\lambda}^{q}(v\cdot \lambda)\not\subset (\mathrm{U_{q}}(\mathfrak{b}))V_{\lambda}^{q}(w\cdot \lambda),$ then we have the following corollary of Lemma~\ref{lem1}.
\begin{cor}\label{vnotw}
If $\lambda\in P_{++}$ and $v\not \leq w,$ then 
\begin{equation*}V_{\lambda}^{q}(v\cdot \lambda)\bot (\mathrm{U_{q}}(\mathfrak{b}))V_{\lambda}^{q}(w\cdot \lambda).\end{equation*}
\end{cor}
\subsection{\bf\textsc{The Quantum Group $\mathbb{C}[\mathrm{G]_{q}}$}}
We define $\mathbb{C}[\mathrm{G]_{q}}\subseteq (\mathrm{U_{q}}(\mathfrak{g}))^{*}$ as the subspace of the dual generated by linear functionals of the form
\begin{equation}\label{cgq}
C_{\eta,\xi}^{\lambda}(a):=\langle a\cdot \eta,\xi\rangle, 
\end{equation}
$$
\begin{array}{ccccc}
\text{for} &a\in \mathrm{U_{q}}(\mathfrak{g}),& \eta,\xi\in V_{\lambda}^{q}, & \lambda\in P_{+}.
\end{array}
$$
We let $\Delta_{q},$ $\epsilon_{q}$ and $S_{q}$ respectively be the dual of the product, co-product and antipode of $\mathrm{U_{q}}(\mathfrak{g}).$  Moreover, we define a $*$-involution on $\mathbb{C}[\mathrm{G]_{q}}$ by the formula $$(C_{\eta,\xi}^{\lambda})^{*}(a)=\overline{C_{\eta,\xi}^{\lambda}((\hat S_{q}(a))^{*})}$$ 
We let $\mathbb{C}[\mathrm{G]}=\mathbb{C}[\mathrm{G]_{1}}$ as well as $\Delta=\Delta_{1}$, $\epsilon=\epsilon_{1}$ and $S=S_{1}$ denote the $*$-algebra of regular functions on $\mathrm{G}$ with the usual co-product, co-unit and antipode.
For an irreducible module $V^{q}_{\lambda}$ with an orthonormal basis $\xi_{k},$ $k=1,\dots, m=\dim V_{\lambda}^{q},$ it follows from the definition of $\mathbb{C}[\mathrm{G]_{q}}$ that for $i,k=1,\dots, m,$
\begin{gather}
\Delta_{q}(C_{\xi_{i},\xi_{k}}^{\lambda})=\sum_{j=1}^{m}C_{\xi_{i},\xi_{j}}^{\lambda}\otimes C_{\xi_{j},\xi_{k}}^{\lambda},\label{xi1}\\
\epsilon_{q}(C_{\xi_{i},\xi_{k}}^{\lambda})=C_{\xi_{i},\xi_{k}}^{\lambda}(I)=\langle\xi_{i},\xi_{k}\rangle=\delta_{ik}1.\label{xi2}\\
\begin{array}{ccc}
\sum_{i=1}^{m}(C_{\xi_{i},\xi_{k}}^{\lambda})^{*}C_{\xi_{i},\xi_{j}}^{\lambda}=\delta_{km}I, & k,j=1,\dots, m.\label{xi3}
\end{array}
\end{gather}
To deduce~\eqref{xi3}, consider any $a\in \mathrm{U_{q}}(\mathfrak{g}),$ then it follows from the Hopf algebra axioms that (noticing that $\overline{C_{\eta,\xi}^{\lambda}(a)}=C_{\xi,\eta}^{\lambda}(a^{*})$ and hence $(C_{\eta,\xi}^{\lambda})^{*}(a)=C_{\xi,\eta}^{\lambda}(\hat S_{q}(a))$)
$$
\sum_{i=1}^{m}(C_{\xi_{i},\xi_{k}}^{\lambda})^{*}C_{\xi_{i},\xi_{j}}^{\lambda}(a)=\sum_{i=1}^{m}((C_{\xi_{i},\xi_{k}}^{\lambda})^{*}\otimes C_{\xi_{i},\xi_{j}}^{\lambda})(\hat\Delta_{q}(a))=
$$
$$
=\sum_{i=1}^{m}(C_{\xi_{k},\xi_{i}}^{\lambda}\otimes C_{\xi_{i},\xi_{j}}^{\lambda})((\hat S_{q}\otimes \iota)\hat\Delta_{q}(a))=
$$
$$
=C^{\lambda}_{\xi_{k},\xi_{j}}(\sum \hat S_{q}(a_{(1)})a_{(2)})=C^{\lambda}_{\xi_{k},\xi_{j}}(\hat \epsilon_{q}(a))=\delta_{kj}\hat\epsilon_{q}(a)I,
$$
where $\hat \Delta_{q}(a)=\sum a_{(1)}\otimes a_{(2)}$ in the Sweedler notation.
We define inductively 
$$
\Delta_{q}^{(2)}=\Delta_{q}.
$$
\begin{equation}\label{deltapower}\Delta_{q}^{(n)}=(\underset{\text{$n-1$ terms}}{\underbrace{\Delta_{q}\otimes \iota\otimes\cdots\otimes \iota}})\circ \Delta^{(n-1)}_{q}:\mathbb{C}[\mathrm{G]_{q}}\longrightarrow \underset{\text{$n$ terms}}{\underbrace{\mathbb{C}[\mathrm{G]_{q}}\otimes \cdots \otimes \mathbb{C}[\mathrm{G]_{q}}}}\end{equation}

Notice that by co-associativity $(\Delta_{q}\otimes \iota)\circ \Delta_{q}=(\iota\otimes \Delta_{q})\circ \Delta_{q}$, it does not matter which tensor factor you apply $\Delta_{q}$ to in~\eqref{deltapower}. 
\\

Let us denote by $C(\mathrm{G)_{q}}$ the universal enveloping $C^{*}$-algebra of $\mathbb{C}[\mathrm{G]_{q}}.$ It is known from~\cite{ks} that the universal enveloping $C^{*}$-algebra exists and that the natural homomorphism $\mathbb{C}[\mathrm{G]_{q}}\hookrightarrow C(\mathrm{G)_{q}}$ is injective. Hence we can identify $\mathbb{C}[\mathrm{G]_{q}}$ with its inclusion $\mathbb{C}[\mathrm{G]_{q}}\subseteq C(\mathrm{G)_{q}}.$ Moreover, the co-product can be extended to a $*$-homomorphism $\Delta_{q}:C(\mathrm{G)_{q}}\to C(\mathrm{G)_{q}}\otimes C(\mathrm{G)_{q}}$ (the minimal tensor product), giving a structure of $C(\mathrm{G)_{q}}$ as a compact quantum group in the sense of Woronowicz~\cite{wz2}. We will use the same symbol for a $*$-representation of $\mathbb{C}[\mathrm{G]_{q}}$ as well as its extension to $C(\mathrm{G)_{q}}.$
\\

Recall the special case of $\mathrm{SU_{2}}.$ Let $V^{q}_{\lambda}$ be the unique $2$-dimensional $\mathrm{U_{q}}(\mathfrak{su}_{2})$-module, with basis $\xi_{\lambda}=:\xi_{1},\xi_{-\lambda}=:\xi_{2}$ and let 
$$
\begin{array}{ccc}
t_{ij}=C^{\lambda}_{\xi_{i},\xi_{j}}, &\text{for $i,j=1,2.$}
\end{array}
$$
Then the elements $t_{ij}$ generate $\mathbb{C}[\mathrm{SU}_{2}]_{q}$ as an algebra, and they are subject to the relations
\begin{equation}\label{su2q}
\begin{array}{cc}
\begin{array}{cccc}
t_{11}t_{21}=qt_{21}t_{11}, & t_{11}t_{12}=qt_{12}t_{11}, & t_{12}t_{21}=t_{21}t_{12},
 \end{array}\\
\begin{array}{cccc}
t_{22}t_{21}=q^{-1}t_{21}t_{11}, & t_{22}t_{12}=q^{-1}t_{12}t_{22},
\end{array}\\
\begin{array}{cccc}
t_{11}t_{22}-t_{22}t_{11}=(q-q^{-1})t_{12}t_{21}, & t_{11}t_{22}-q t_{12}t_{21}=1,
\end{array}\\
\begin{array}{cccc}
t^{*}_{11}=t_{22}, & t_{12}^{*}=-q t_{21}
\end{array}
\end{array}
\end{equation}
Moreover, the relations~\eqref{su2q} determine $\mathbb{C}[\mathrm{SU_{2}}]_{q},$ in the sense that this algebra is isomorphic to the universal $*$-algebra with generators $\hat t_{ij}$, $i,j=1,2,$ satisfying the relations~\eqref{su2q}.  
\\

Given two $*$-representations $\pi_{1},\pi_{2},$ such that $\pi_{i}:C(\mathrm{G)_{q}}\to \mathcal{B}(\Hh_{i})$ for $i=1,2,$ we can define the tensor product using the co-multiplication as
\begin{equation}\label{hopftensor}
\pi_{1}\boxtimes\pi_{2}:=(\pi_{1}\otimes \pi_{2})\circ \Delta_{q}:C(\mathrm{G)_{q}}\longrightarrow  \mathcal{B}(\Hh_{1})\otimes \mathcal{B}(\Hh_{2})\subseteq  \mathcal{B}(\Hh_{1}\otimes \Hh_{2})
\end{equation}
where $\otimes$ denotes the minimal tensor product between $C^{*}$-algebras. We will also use $\otimes$ to denote the algebraic tensor product; it will always be clear from context which one we use (e.g. we will never take the algebraic tensor product between two $C^{*}$-algebras).
\subsection{\bf\textsc{Representation Theory of $\mathbb{C}[\mathrm{G]_{q}}$}}
Recall a $*$-representation $\Pi_{q}$ of $\mathbb C[\mathrm{SU_2]_q}$  defined in the following way:
let $C_q,S,d_q:\ell^2(\mathbb Z_+)\to\ell^2(\mathbb Z_+)$ be the operators defined on the natural orthonormal basis $\{e_{j}\}_{j\in \mathbb{Z}_{+}}$ as follows:
\begin{equation}\label{SC}
Se_n=e_{n+1},\quad C_q e_n=\sqrt{1-q^{2n}}e_n,\quad d_q e_n=q^ne_n.
\end{equation}

Then the map 
\begin{equation}\label{SC1}
\Pi_{q}(t_{1,1})=S^*C_q,\  \Pi_{q}(t_{1,2})=q d_q,\  \Pi_{q}(t_{2,1})=-d_q,\  \Pi_{q}(t_{2,2})=C_q S
\end{equation}
extends to a $*$-representation of $\mathbb C[\mathrm{SU_2]_q}.$ Let $C^{*}(S)\subseteq \mathcal{B}(\ell^{2}(\mathbb{Z}_{+}))$ be the $C^{*}$-algebra generated by $S$ (so that $C^{*}(S)$ is equal to the Toeplitz algebra). From the expressions~\eqref{SC} for $C_{q}$ and $d_q$ it is easy to see that $C_{q},d_{q}\in C^{*}(S)$ and hence that $$\Pi_{q}(\mathbb C[\mathrm{SU_2]_q})\subseteq C^{*}(S).$$
We recall the $*$-representation theory for $\mathbb{C}[\mathrm{G]_{q}},$ $q\in (0,1),$ due to Soibelman~\cite{ks}: For each simple root $\alpha_{i}\in \Omega$ we get an injective Hopf $*$-homomorphism $\mathrm{U_{q_{i}}}(\mathfrak{su}_{2})\to \mathrm{U_{q}}(\mathfrak{g})$ that dualizes to a surjective Hopf $*$-homomorphism
$$
\begin{array}{cccc}
\varsigma_{i}^{q}:\mathbb{C}[\mathrm{G]_{q}}\longrightarrow \mathbb{C}[\mathrm{SU_{2}]_{q_{i}}},  & i=1,\dots, n
\end{array}
$$
(this is true also for $q=1,$ we write $\varsigma_{i}=\varsigma_{i}^{1}$). For $i=1,\dots, n,$ we define $*$-representations $$\pi_{i}^{q}:=\Pi_{q_{i}}\circ \varsigma_{i}^{q}:\mathbb{C}[\mathrm{G]_{q}}\longrightarrow C^{*}(S)\subseteq \mathcal{B}(\ell^{2}(\mathbb{Z}_{+})).$$
Let $w\in W$, with a reduced presentation $w=s_{j_{1}}\cdots s_{j_{m}}$ (and hence $m=\elll(w)$), and define 
\begin{equation}
\pi_{w}^{q}:=\pi_{j_{1}}^{q}\boxtimes \cdots \boxtimes \pi_{j_{m}}^{q}:\mathbb{C}[\mathrm{G]_{q}}\longrightarrow C(S)^{\otimes \elll(w)}\subseteq\mathcal{B}(\ell^{2}(\mathbb{Z}_{+})^{\otimes \elll(w)})
\end{equation}
when $e\in W$ is the identity element, we let $\pi_{e}=\epsilon_{q}$ (corresponding to the empty reduced presentation).
From~\cite{ks}, we know that $\pi_{w}^{q}$ does not depend on the reduced decomposition, in the sense that if we have two reduced presentations $w=s_{j_{1}}\cdots s_{j_{m}}=s_{j'_{1}}\cdots s_{j'_{m}},$ then the two corresponding $*$-representations
$$
\begin{array}{cccc}
\pi_{j_{1}}^{q}\boxtimes \cdots \boxtimes \pi_{j_{m}}^{q}, & \pi_{j'_{1}}^{q}\boxtimes \cdots \boxtimes \pi_{j'_{m}}^{q}
\end{array}
$$
are unitarily equivalent. For simplicity, let us write $$\Hh_{w}:=\ell^{2}(\mathbb{Z}_{+})^{\otimes \elll(w)}.$$ Thus $\pi_{w}^{q}$ is a $*$-representation $\mathbb{C}[\mathrm{G]_{q}}\to \mathcal{B}(\Hh_{w}).$ We have a subgroup $\mathrm T\subseteq G$ of a maximal torus, corresponding to the real sub-algebra $\mathfrak{t}\subseteq \mathfrak{g}.$ Let $\omega_{i},$ $i=1,\dots,n,$ be the fundamental weights for $\mathfrak g.$ We have an isomorphism  $\mathrm T\cong \mathbb{T}^{n},$ given by
\begin{equation}\label{TnT}
\begin{array}{cccc}
t=e^{x}\in \mathrm T\mapsto (e^{\omega_{1}(x)},\dots,e^{\omega_{n}(x)})\in \mathbb{T}^{n},
\end{array}
\end{equation}
where $x\in\mathfrak t.$ For every $t\in\mathrm T,$ we have mutually non-equivalent one-dimensional $*$-representations $\chi_{t}:\mathbb{C}[\mathrm{G]_{q}}\to \mathbb{C},$ such that, for $s,t\in\mathrm T,$ we have $\chi_{s}\boxtimes \chi_{t}=\chi_{st}$ and $\chi_{t}\boxtimes \chi_{t^{-1}}=\chi_{1}=\epsilon_{q}.$ By~\cite{nt}, we can for every $t\in \mathrm T$ associate a unitary operator $U_{t}\in \prod_{\lambda \in P_{+}}\mathcal{B}(V^{q}_{\lambda}),$ such that if $x\in \mathfrak t$ satisfies $t=e^{x},$ then for all $\lambda \in P_{+}$ and $\gamma\in P$
\begin{gather}
\begin{array}{cccc}
\langle U_{t}\eta, \xi \rangle=\chi_{t}(C_{\eta,\xi}^{\lambda}), & \eta,\xi\in V_{\lambda}^{q},
\end{array}\label{maxtor}\\
\begin{array}{ccc}
U_{t}\eta=e^{\gamma(x)}\eta, & \eta\in V_{\lambda}^{q}(\gamma).
\end{array}\label{maxtor2}
\end{gather}

The following theorem characterizes all irreducible $*$-representations of $\mathbb{C}[\mathrm{G]_{q}}$ (up to unitary equivalence). 
\begin{thm}[{Theorem~6.2.7 in~\cite{ks}}]

\begin{enumerate}[(i)]
\item For $w\in W$ and $t\in \mathrm T,$ the $*$-representations $\pi_{w}^{q}\boxtimes \chi_{t}$ are irreducible and mutually non-equivalent, and
\item every irreducible $*$-representation $\pi$ of $\mathbb{C}[\mathrm{G]_{q}}$ is unitarily equivalent to some $\pi_{w}\boxtimes \chi_{t},$ $w\in W,$ $t\in \mathrm T.$
\end{enumerate}
\end{thm}

\section{\bf\textsc{Properties of $C(\mathrm{G)_{q}}$ under $*$-Representations}}
\subsection{\bf\textsc{Paths in the Weyl Group and Subsets of $\mathrm T$}}
For elements $v,w\in W,$ we write $v\lhd w$ to mean that
\begin{enumerate}[(i)]
\item $v<w$ in the Bruhat order,
\item there is no $r\in W,$ such that $v<r<w.$
\end{enumerate}
By Theorem~2.2.6 in~\cite{abfb}, this means that there is a $\alpha\in \Phi_{+}$ such that $v\alpha=w$ and $l(v)=l(w)-1.$ Keeping with the established terminology, we also say that $w$ \textit{covers} $v$ if $v\lhd w.$
In general, we write $v\lhd^{(k)} w$ if we have elements $r_{1},\dots,r_{k-1}\in W$ such that 
$$
v\lhd r_{1}\lhd\dots\lhd r_{k-1}\lhd w.
$$
It is a property of $W$ that every chain 
$$
v\lhd r \lhd \dots\lhd w
$$
must have the same length and hence that the relation $\lhd^{(k)}$ is actually well-defined for $k>1$ (see Theorem~2.2.6 in~\cite{abfb}).
\\

Let $v,w\in W.$ If $v\lhd w,$ we also write $v\overset{\gamma}{\leadsto}w,$ for a $\gamma\in \Phi_{+},$ if $v\gamma =w.$ In general, if $v\leq w,$ then we say that we have a path from $v$ to $w$
\begin{equation}\label{paths}
v=v_{1}\overset{\gamma_{1}}{\leadsto}v_{2}\overset{\gamma_{2}}{\leadsto}\dots\overset{\gamma_{m-1}}{\leadsto}v_{m}\overset{\gamma_{m}}{\leadsto}v_{m+1}=w
\end{equation}
if $v_{j}\lhd v_{j+1}$ and $v_{j}\gamma_{j+1}=v_{j+1}$ for $j=1,\dots,m.$ Clearly, every path from $v$ to $w$ has the same length $m=l(w)-l(v).$ 
We write~\eqref{paths} as the composition of paths
\begin{equation}
\begin{array}{cccc}
v\overset{\gamma}{\leadsto} w, & \text{for $\gamma=\gamma_{1}\circ\dots\circ \gamma_{m}$}
\end{array}
\end{equation}
to indicate that we have a specific path between $v$ and $w.$
\\

For each path $v\overset{\gamma}{\leadsto}w$ we associate a closed connected subgroup $\mathrm T_{\gamma}\subseteq \mathrm T$ by taking the exponential of the real subspace of $\mathfrak{t}$ spanned by $h_{\gamma_{i}}$ for $i=1,\dots,m.$ For the path $v \overset{\gamma_{1}}{\leadsto}r  \overset{\gamma_{2}}{\leadsto} w$ and $\gamma =\gamma_2\circ\gamma_1,$ we have \begin{equation}\label{gamma}\mathrm T_{\gamma_{2}}\mathrm  T_{\gamma_{1}}:=\{ts:t\in \mathrm T_{\gamma_{2}},s\in \mathrm T_{\gamma_{1}}\}=\mathrm T_{\gamma}.\end{equation} Let $\mathrm T_{v}^{w}$ be the union of all the subgroups of $T$ generated by paths from $v$ to $w.$ From~\eqref{gamma}, it follows that we have the following multiplicative property:
If $v\leq r\leq w$ then $$\mathrm T_{v}^{r}\mathrm T_{r}^{w}\subseteq \mathrm T_{v}^{w}.$$
Clearly, it follows from~\eqref{paths} that for $v<w,$ we have
\begin{equation}\label{mult}
\mathrm T_{v}^{w}=\underset{{v<r\lhd w}}{\cup} \mathrm T_{v}^{r}\mathrm T_{r}^{w},
\end{equation}
where the union ranges over all $r\in W$ such that $v<r\lhd w.$ For later use, we also note the following special case of~\eqref{gamma}: if $v\overset{\gamma}{\leadsto}r\lhd w$ with $r \alpha= w$  then
\begin{equation}\label{gamma1}
\mathrm T_{\gamma}\mathrm T_{r}^{w}=\mathrm T_{\alpha\circ \gamma}.
\end{equation}
\subsection{\bf\textsc{Ideals and Quotients}}
 We will use the following results from~\cite{nt} (though stated in a less general fashion in order to suit our purposes).
\begin{thm}[Theorem~$4.1$ $(ii)$ in~\cite{nt}]\label{ntthm}
Let $\sigma\in W$ and $\mathrm Y\subseteq \mathrm T.$ For any $r\in W$ and $t\in T,$ the kernel of the representation $\pi_{r}^{q}\boxtimes \chi_{t}$ contains the intersection of the kernels of the representations $\pi_{\sigma}^{q}\boxtimes \chi_{s},$ $s\in Y$ of $C(\mathrm{G)_{q}}$ if and only if $r\leq \sigma$ and $t\in \overline{\mathrm Y}\mathrm T_{r}^{\sigma}.$
\end{thm}
\begin{lem}[Lemma~$4.5$ in~\cite{nt}]\label{ntlem}
Let $t\in \mathrm T$ and let $w\in W$. Assume $x\in C(\mathrm{G)_{q}}$ is such that $(\pi_{v}^{q}\boxtimes \chi_{s})(x)=0$ for all $v\in W$ such that $v<w$ and $s\in t \mathrm T_{v}^{w},$ then 
\begin{equation}\label{compact}
(\pi_{w}^{q}\boxtimes\chi_{t})(x)\in \K_{w}.
\end{equation}
\end{lem}
Recall the definition of $C_{\eta,\xi}^{\lambda}\in \mathbb{C}[\mathrm{G]_{q}},$ given by~\eqref{cgq}. To avoid multiple subscripts, let us write $w\cdot \lambda$ in place of $\xi_{w\cdot \lambda}$ in~\eqref{cgq}. Thus, for example, we write $C^{\lambda}_{w\cdot\lambda,\lambda}$ instead of $C^{\lambda}_{\xi_{w\cdot \lambda},\xi_{\lambda}}.$
\begin{lem}[Lemma~2.3 in~\cite{nt}]\label{lem2}
Let $w\in W$ and $\lambda\in P_{+}.$ 
\begin{enumerate}[(i)]
\item $\pi_{w}^{q}(C_{w\cdot\lambda,\lambda}^{\lambda})$ is a compact contractive diagonalizable operator with zero kernel, and the vector $e_{0}^{\otimes \elll(w)}\in \Hh_{w}$ is its only eigenvector (up to scalar) with an eigenvalue of absolute value $1$.
\item If $\zeta\in V^{q}_{\lambda}$ is orthogonal to $(\mathrm{U_{q}}(\mathfrak{b}))V^{q}_{\lambda}(w\cdot \lambda),$ then $$\pi_{w}^{q}(C_{\zeta,\lambda}^{\lambda})=0.$$
\end{enumerate}
\end{lem} 

When $q=1,$ we get a Hopf $*$-algebra homomorphism $\tau_{1}=\tau:\mathbb{C}[\mathrm{G]}\to \mathbb{C}[\mathrm{T]}$ by restriction to the subgroup $\mathrm{T\subseteq G}$ (here $\mathbb{C}[\mathrm{T]}$ is the Hopf $*$-algebra of trigonometric polynomials on $\mathrm T$). For general $q\in (0,1),$ we define a surjective Hopf $*$-algebra homomorphism $\tau_{q}:\mathbb{C}[\mathrm{G]_{q}}\to \mathbb{C}[\mathrm{T]},$ that extends to a homomorphism of compact quantum groups $\tau_{q}: C(\mathrm{G)_{q}}\to C(\mathrm T)$. We define $\tau_{q}$ in the following way:
The compact operators $\K\subseteq \mathcal{B}(\ell^{2}(\mathbb{Z}_{+}))$ is a $*$-ideal in $C^*(S),$ and it is well known that we have the isomorphism $p:C^*(S)/\K\mapsto C(\mathbb{T})$ such that $S^{*}+\K\mapsto z$ (here $z\in C(\mathbb{T})$ is the coordinate function). Moreover, it is easy to see that we actually have a homomorphism of Hopf $*$-algebras $\beta_{q}:\mathbb C[\mathrm{SU_2]_q}\to \mathbb{C}[\mathbb T]\subseteq C(\mathbb{T})$ that factors as
\begin{equation}\label{su2}
\mathbb C[\mathrm{SU_2]_q}\overset{\Pi_{q}}{\longrightarrow} C^*(S) \overset{p}{\longrightarrow}  C^*(S)/\K\cong C(\mathbb T)
\end{equation}
and such that $\beta_{q}(t_{11}^{q})=z.$ Consider now the $*$-homomorphism 
\begin{equation}\label{tauq}
\tau_{q}:\mathbb C[\mathrm{G]_q}\overset{\pi_{1}^{q}\boxtimes\cdots \boxtimes \pi_{n}^{q}}\longrightarrow C^*(S)^{\otimes n}\overset{p\otimes\cdots\otimes p}\longrightarrow C(\mathbb T)^{\otimes n}\cong C(\mathrm T),
\end{equation}
with the isomorphism $C(\mathbb T)^{\otimes n}\cong C(\mathrm T)$ induced by the isomorphism $\mathbb{T}^{n}\cong \mathrm T$ given by~\eqref{TnT}. By using~\eqref{su2}, we can also factor $\tau_{q}$ as 
\begin{equation}\label{tauq2}
\mathbb{C}[\mathrm{G]_{q}}\overset{(\varsigma_{1}^{q}\otimes \cdots\otimes \varsigma_{n}^{q})\circ \Delta_{q}^{(n)}}\longrightarrow \mathbb{C}[\mathrm{SU_{2}]_{q_{1}}}\otimes \cdots \otimes \mathbb{C}[\mathrm{SU_{2}]_{q_{n}}}\overset{\beta_{q_{1}}\otimes \cdots\otimes \beta_{q_{n}}}\longrightarrow \mathbb{C}[\mathbb T]^{\otimes n}\cong \mathbb{C}[\mathrm T].
\end{equation}
If we consider $\mathbb{C}[\mathrm{SU_{2}]_{q_{1}}}\otimes \cdots \otimes \mathbb{C}[\mathrm{SU_{2}]_{q_{n}}}$ as a tensor product of Hopf $*$-algebras, hence also a Hopf $*$-algebra, then it is easy to check that the $*$-homomorphisms in~\eqref{tauq2} are actually Hopf $*$-algebra morphisms. Thus, $\tau_{q}$ is a morphisms of Hopf $*$-algebras.
\begin{lem}\label{maxtau}
\begin{enumerate}[(i)]
\item The $*$-homomorphism $\tau_{q}:C(\mathrm{G)_{q}}\to C(\mathrm T)$ is surjective,

\item every $\chi_{t},$ $t\in \mathrm T,$ factors as 
$$
\chi_{t}:C(\mathrm{G)_{q}}\overset{\tau_{q}}\longrightarrow C(\mathrm T)\overset{ev_{t}}\longrightarrow \mathbb C,
$$
\item for $\eta\in V_{\lambda}(\gamma_{1}),$ $\xi \in V_{\lambda}(\gamma_{2}),$ we have $\tau_{q}(C_{\eta,\xi}^{\lambda})=0$ unless $\gamma_{1}=\gamma_{2}$ and $\langle \eta,\xi\rangle\neq 0,$
\item let $\omega_{i},$ $i=1,\dots, n,$ be the fundamental weights, then the set $\tau_{q}(C_{\omega_{i},\omega_{i}}^{\omega_{i}}),$ $i=1,\dots,n,$ generates $C(\mathrm T)$ as a $C^{*}$-algebra.
\end{enumerate}
\end{lem}
\begin{proof}
Clearly, $(i)$ follows from $(iv)$ or $(ii).$ We have that every one-dimensional $*$-representation of $\mathbb{C}[\mathrm{G]_{q}}$ is of the form $\chi_{t},$ for some $t\in \mathrm T,$ and different weight spaces are orthogonal. Hence we get $(iii)$ from~\eqref{maxtor} and~\eqref{maxtor2} by point evaluation. Note that the same argument shows that $(iii)$ holds for \textit{any} $*$-homomorphism from $\mathbb{C}[\mathrm{G]_{q}}$ to a commutative $C^{*}$-algebra.  

By Lemma~$2.2.1$ in~\cite{ks}, the elements of the form $C^{\lambda}_{\eta,\lambda},$ $\lambda\in P_{+},$ $\eta\in V^{q}_{\lambda}$ generates $\mathbb{C}[\mathrm{G]_{q}}$ as a $*$-algebra. Thus it follows from $(iii)$ that the elements in $C(\mathrm T)$ of the form $\tau_{q}(C_{\lambda,\lambda}^{\lambda}),$ $\lambda\in P_{+}$ generate the image of $\tau_{q}.$ As every $\lambda\in P_{+}$ is a linear combination of the fundamental weights, it follows from~\eqref{maxtor2} that the set $\{\tau_{q}(C_{\omega_{i},\omega_{i}}^{\omega_{i}}):i=1,\dots,n\}$ generate the image of $\tau_{q}.$ Thus $(ii)\Rightarrow (i)\Rightarrow (iv).$ 
To prove $(ii),$ we establish a one-to-one correspondence between one-dimensional $*$-representations and point evaluations of $\tau_{q}.$ By (\cite{ksch}, Theorem~$14,$ Section 6.1.5), the set $F^{i}K^{j}E^{k},$ $j\in \mathbb Z$ $i,k\in \mathbb{Z}_{+},$ is a basis for $\mathrm{U_{q}}(\mathfrak{su}_{2}).$ Using this, and keeping in mind that $\xi_{\omega_{i}}$ is a heighest weight vector, we obtain
\begin{equation}\label{omegai}
\varsigma_{i}^{q}(C_{\omega_{j},\omega_{j}}^{\omega_{j}})=
\begin{cases}
t_{11}^{q_{i}}, & \text{if $i=j$}\\
I, & \text{for $i\neq j.$}
\end{cases}
\end{equation}
It then follows that $(\beta_{i}\circ \varsigma^{q}_{i})(C_{\omega_{j},\omega_{j}}^{\omega_{j}})=z$ for $i=j$ and $(\beta_{i}\circ \varsigma^{q}_{i})(C_{\omega_{j},\omega_{j}}^{\omega_{j}})=I$ if $i\neq j.$ If we extend $\xi_{1}=\xi_{\omega_{i}}$ by $\xi_{2},\dots,\xi_{m}$ to an orthonormal basis for $V_{\omega_{i}}^{q},$ we then get from~\eqref{xi1} and our comment after the proof of (iii), that if we let $z_{i}$ be the i'th coordinate function of $\mathbb{T}^{n},$ then

\begin{equation}
\begin{array}{ccc}
\tau_{q}(C_{\omega_{i},\omega_{i}}^{\omega_{i}})=z_{i}, & i=1,\dots, n,
\end{array}
\end{equation}
and from this it follows that the range of $\tau_{q}$ is dense in $C(\mathrm{T}).$ By~\eqref{maxtor2}, if $t=e^{x},$ $x\in \mathfrak t,$ then we have $\chi_{t}(C^{\omega_{j}}_{\omega_{j},\omega_{j}})=e^{\omega_{j}(x)}$ for $j=1,\dots,n.$ It thus follows that if we use the identification $\mathrm T\cong \mathbb T^{n}$ given by~\eqref{TnT}, then $t=(e^{\omega_{1}(x)},\dots, e^{\omega_{n}(x)})$ and hence
$$
\begin{array}{ccc}
\chi_{t}(C^{\omega_{j}}_{\omega_{j},\omega_{j}})=e^{\omega_{j}(x)}=\ev_{t}(z_{j})=(\ev_{t}\circ \tau_{q})(C^{\omega_{j}}_{\omega_{j},\omega_{j}}), & j=1,\dots, n.
\end{array}
$$
\end{proof}
Let $\mathbb{C}[\mathrm{G]_{q}^{\textsc{inv}}}\subseteq \mathbb{C}[\mathrm{G]_{q}}$ denote the $*$-subalgebra of elements invariant under the left-right-action of the maximal torus $\mathrm T.$ By definition, this is the subset $x\in\mathbb{C}[\mathrm{G]_{q}},$ such that for every $t\in \mathrm T,$ we have $L_{t}(x)=x=R_{t}(x),$ where $L_{t}$ and $R_{t}$ are given by
\begin{equation}\label{laction}
\begin{array}{cc}
L_{t}:\mathbb{C}[\mathrm{G]_{q}}\overset{\Delta_{q}}{\longrightarrow} \mathbb{C}[\mathrm{G]_{q}}\otimes \mathbb{C}[\mathrm{G]_{q}}\overset{\tau_{q}\otimes \iota}{\longrightarrow} C(\mathrm T)\otimes \mathbb{C}[\mathrm{G]_{q}}\overset{ ev_{t}\otimes\iota}{\longrightarrow} \mathbb{C}[\mathrm{G]_{q}} & \text{(Left-action).}
\end{array}
\end{equation}
\begin{equation}\label{raction}
\begin{array}{ccc}
R_{t}:\mathbb{C}[\mathrm{G]_{q}}\overset{\Delta_{q}}{\longrightarrow} \mathbb{C}[\mathrm{G]_{q}}\otimes \mathbb{C}[\mathrm{G]_{q}}\overset{\iota\otimes \tau_{q}}{\longrightarrow} \mathbb{C}[\mathrm{G]_{q}}\otimes C(\mathrm T)\overset{\iota\otimes ev_{t}}{\longrightarrow} \mathbb{C}[\mathrm{G]_{q}} & \text{(Right-action)}
\end{array}
\end{equation}
Clearly $\mathbb{C}[\mathrm{G]_{q}^{\textsc{inv}}}$ is a $*$-subalgebra of $\mathbb{C}[\mathrm{G]_{q}}.$
\begin{lem}\label{lem3}
For every $w\in W,$ there exists $\Upsilon_{w}\in \mathbb{C}[\mathrm{G]_{q}^{\textsc{inv}}}$ such that 
\begin{enumerate}[(i)]
\item $\pi_{w}^{q}(\Upsilon_{w})\in \mathcal{B}(\Hh_{w})$ is a compact contractive positive operator with dense range,
\item $e_{0}^{\otimes \elll(w)}\in \Hh_{w}$ is the only eigenvector of $\pi_{w}^{q}(\Upsilon_{w})$ (up to a scalar multiple) with eigenvalue $1,$
\item $\pi_{v}^{q}(\Upsilon_{w})\neq 0$ if and only if $v\geq w .$
\end{enumerate}
\end{lem}
\begin{proof}
Take any $\lambda\in P_{++}$ and consider $C_{w\cdot \lambda,\lambda}^{\lambda}.$ By combining Corollary~\ref{vnotw} with Lemma~\ref{lem2} it follows that $\pi_{v}^{q}(C_{w\cdot\lambda,\lambda}^{\lambda})=0$ for any $v\not \geq w.$ If $v\geq w,$ then as $1\in \mathrm T_{w}^{v}$ (the identity of $\mathrm T$) it follows from Theorem~\ref{ntthm} that we have $\ker \pi_{v}^{q}\subseteq \ker \pi_{w}^{q}.$ As $\pi_{w}^{q}(C_{w\cdot\lambda,\lambda}^{\lambda})\neq 0,$ it thus follows that also $\pi_{v}^{q}(C_{w\cdot \lambda,\lambda}^{\lambda})\neq 0.$ If we extend $\xi_{\lambda},\xi_{w\cdot\lambda}$ to an orthonormal basis for $V^{q}_{\lambda},$ then we get from~\eqref{xi1} and Lemma~\ref{maxtau} $(iii)$ that
\begin{equation}\label{lraction}
\begin{array}{cccc}
L_{t}(C_{w\cdot\lambda,\lambda}^{\lambda})=\chi_{t}(C_{w\cdot\lambda,w\cdot\lambda}^{\lambda})\cdot C_{w\cdot\lambda,\lambda}^{\lambda}, & R_{t}(C_{w\cdot\lambda,\lambda}^{\lambda})=\chi_{t}(C_{\lambda,\lambda}^{\lambda})\cdot C_{w\cdot\lambda,\lambda}^{\lambda},& t\in \mathrm T.
\end{array}
\end{equation}
Let us now define $\Upsilon_{w}:=(C_{w\cdot\lambda,\lambda}^{\lambda})^{*}C_{w\cdot\lambda,\lambda}^{\lambda}.$ As $L_{t}$ and $R_{t}$ are $*$-automorphisms for all $t\in\mathrm T,$ it follows from~\eqref{lraction} that $\Upsilon_{w}\in \mathbb{C}[\mathrm{G]_{q}^{\textsc{inv}}}.$ Positivity follows from the definition of $\Upsilon_{w},$ and the other claims in $(i)$ and $(iii)$ follow by Lemma~\ref{lem2} $(i)$.  To see $(ii),$ note that
\begin{equation}\label{calc1}
\langle \pi_{w}^{q}(\Upsilon_{w})e_{0}^{\otimes \elll(w)},e_{0}^{\otimes \elll(w)}\rangle=\|\pi_{w}^{q}(C_{w\cdot\lambda,\lambda}^{\lambda})e_{0}^{\otimes \elll(w)} \|^{2}= \|e_{0}^{\otimes \elll(w)} \|^{2}=1.
\end{equation}
 As $\pi_{w}^{q}(\Upsilon_{w})$ is a positive contraction, it follows that $e_{0}^{\otimes \elll(w)}$ must be an eigenvector with eigenvalue $1$. To see that $e_{0}^{\otimes \elll(w)}$ is the only eigenvector (up to a scalar multiple) of $\pi_{w}^{q}(\Upsilon_{w})$ with eigenvalue $1,$ notice that~\eqref{calc1} gives that $e_{0}^{\otimes \elll(w)}$ is also an eigenvector for $\pi_{w}^{q}(C_{w\cdot\lambda,\lambda}^{\lambda})^{*}$ and thus the subspace generated by $e_{0}^{\otimes \elll(w)}$ is actually reducing $\pi_{w}^{q}(C_{w\cdot\lambda,\lambda}^{\lambda}).$ By $(i)$ of Lemma~\ref{lem2}, it follows that $\pi_{w}^{q}(C_{w\cdot\lambda,\lambda}^{\lambda})$ must have norm strictly less than $1$ when restricted to the orthogonal complement of $e_{0}^{\otimes \elll(w)}$ (otherwise there would be another eigenvector orthogonal to $e_{0}^{\otimes \elll(w)}$ with an eigenvalue of absolute value $1$). The same then holds for $\pi_{w}^{q}(\Upsilon_{w})=\pi_{w}^{q}(C_{w\cdot\lambda,\lambda}^{\lambda})^{*}\pi_{w}^{q}(C_{w\cdot\lambda,\lambda}^{\lambda}).$ 
\end{proof}
\begin{defn}
Let us call a $*$-homomorphism of $\mathbb{C}[\mathrm{G]_{q}}$ a \textit{commutative} $*$-representation if the image sits inside a commutative $C^{*}$-algebra.
\end{defn}
Let $\chi:\mathbb{C}[\mathrm{G]_{q}}\to C(\mathrm X)$ be a commutative $*$-representation. As every one-dimensional $*$-representation of $\mathbb{C}[\mathrm{G]_{q}}$ factors throught the commutative $C^{*}$-algebra $C(\mathrm T),$ it follows that the commutative $*$-representation $\chi$ factors as 
$$
\chi=\zeta\circ \tau_{q}
$$
$$
\mathbb{C}[\mathrm{G]_{q}}\overset{\tau_{q}}{\to} C(\mathrm T)\overset{\zeta}{\to} C(\mathrm X)
$$
for a unique $*$-homomorphism $\zeta.$
\begin{defn}
Let $$\begin{array}{ccc}\chi^{q}:\mathbb{C}[\mathrm{G]_{q}}\to C(\mathrm X), & q\in (0,1)\end{array}$$ be a family of $*$-homomorphisms, where $X$ is a fixed compact Hausdorff space. We say that the $*$-homomorphisms $\{\chi^{q}\}_{q\in (0,1)}$ are \textit{$q$-independent} if in the factorization $$\chi^{q}=\zeta^{q}\circ \tau_{q},$$
$$
\mathbb{C}[\mathrm{G]_{q}}\overset{\tau_{q}}{\to}  C(\mathrm T) \overset{\zeta^{q}}{\to} C(\mathrm X),
$$
we have $\zeta^{q}=\zeta^{s}$ for all $q,s\in (0,1).$
\end{defn}
\begin{defn}
\begin{enumerate}[(a)]
\item For $q\in (0,1),$ let $\Bb_{w}^{q}\subseteq  \mathcal{B}(\Hh_{w})$ be the closure of the image $\pi_{w}:\mathbb{C}[\mathrm{G]_{q}}\rightarrow \mathcal{B}(\Hh_{w}).$
\item If $\chi:\mathbb{C}[\mathrm{G]_{q}}\rightarrow C(\mathrm X)$ is a commutative $*$-representation, then let $B_{w,\chi}^{q}$ be the closure of the image of $\pi_{w}\boxtimes \chi:\mathbb{C}[\mathrm{G]_{q}}\rightarrow B_{w}^{q}\otimes C(\mathrm X).$ 
\end{enumerate}
\end{defn}
\begin{prop}\label{prop2}
Let $w\in W.$ Let $\K_{w}\subseteq \mathcal{B}(\Hh_{w})$ be the space of compact operators and let 
$$
p_{w}: \mathcal{B}(\Hh_{w})\longrightarrow \mathcal {Q}(\Hh_{w})
$$
be the quotient map to the Calkin algebra. For every $v\lhd w,$ there exists a subset $\mathrm T_{v}^{w}\subseteq \mathrm T$ and a commutative $q$-independent $*$-representation $\chi_{v}^{w}:\mathbb{C}[\mathrm{G]_{q}}\to C(\mathrm T_{v}^{w}),$ such that the map
\begin{equation}\label{isomorph}
\begin{array}{ccc}\eta_{w}^{q}:\pi_{w}^{q}(x)+\K_{w}\mapsto  \sumrep{}(\pi_{v}^{q}\boxtimes \chi_{v}^{w})(x), & \text{for $x\in \mathbb{C}[\mathrm{G]_{q}},$}\end{array}
\end{equation}
(where the sum ranges over all $v\in W$ covered by $w$) determines an isomorphism
\begin{equation}\label{lower}
\eta_{w}^{q}:\Bb^{q}_{w}/\K_{w}\longrightarrow\overline{ \sumrep{}(\pi_{v}^{q}\boxtimes \chi_{v}^{w})(\mathbb{C}[\mathrm{G]_{q}})}.
\end{equation}

\end{prop}
We will postpone the proof of Proposition~\ref{prop2} until after Lemma~\ref{lem5}. 

\begin{lem}\label{lem4}
If $x\in \mathbb{C}[\mathrm{G]_{q}^{\textsc{inv}}},$ then for any $w\in W$ and any commutative $*$-representation $\chi,$ we have
$$
(\pi_{w}\boxtimes \chi)(x)=\pi_{w}(x)\otimes I.
$$
\end{lem}
\begin{proof}
By~\eqref{raction} and~\eqref{laction}, the left and right-actions by $t\in \mathrm T$ on $\mathbb{C}[\mathrm{G]_{q}}$ are respectively given as the compositions
$$
\mathbb{C}[\mathrm{G]_{q}}\overset{(\tau_{q}\otimes\iota)\circ \Delta_{q}}{\to} C(\mathrm T)\otimes\mathbb{C}[\mathrm{G]_{q}}\overset{\ev_{t}\otimes \iota}{\to} \mathbb{C}[\mathrm{G]_{q}},
$$
$$
\mathbb{C}[\mathrm{G]_{q}}\overset{(\iota\otimes\tau_{q})\circ \Delta_{q}}{\to} \mathbb{C}[\mathrm{G]_{q}}\otimes C(\mathrm T)\overset{ \iota\otimes \ev_{t}}{\to} \mathbb{C}[\mathrm{G]_{q}}.
$$
Let $x\in \mathbb{C}[\mathrm{G]_{q}^{\textsc{inv}}}.$ Clearly, it follows $L_{t}(x)=R_{t}(x)=x$ for all $t\in \mathrm T$ if and only if
$$
\begin{array}{ccc}
(\iota\otimes\tau_{q})\circ \Delta_{q}(x)=x\otimes I, &(\tau_{q}\otimes\iota)\circ \Delta_{q}(x)=I\otimes x.
\end{array}
$$
The statement now follows from the fact that any $*$-homomorphism $$\chi:C(\mathrm{G)_{q}}\rightarrow C(\mathrm X)$$ factors as $\chi=\zeta\circ\tau_{q}$ for a unique $*$-representation $\zeta:C(\mathrm T)\to C(\mathrm X).$
\end{proof}
\begin{lem}\label{lem5}
Let $\chi:C(\mathrm{G)_{q}}\rightarrow C(\mathrm X)$ be a $*$-homomorphism such that we have $\chi(C(\mathrm{G)_{q}})=C(\mathrm X).$ Then for any $w\in W,$ 
\begin{equation}\label{subset}
\K_{w}\otimes C(\mathrm X)\subset \Bb_{w,\chi}^{q}.
\end{equation}
\end{lem}

\begin{proof}
If we, for any $\lambda\in P_{++},$ extend $\xi_{\lambda},\xi_{w\cdot \lambda}$ to an orthonormal basis of $V_{\lambda}^{q},$ then we get from~\eqref{xi1} and Lemma~\ref{maxtau} $(iii)$ that $$(\pi_{w}^{q}\boxtimes \tau_{q})\left(C_{w\cdot\lambda,\lambda}^{\lambda}\right)=\pi_{w}^{q}\left(C_{w\cdot\lambda,\lambda}^{\lambda}\right)\otimes \tau_{q}\left(C_{\lambda,\lambda}^{\lambda}\right).$$ By Lemma~\ref{lem3} and Lemma~\ref{lem4}, if $p_{0}$ is the orthogonal projection onto $e_{0}^{\otimes \elll(w)},$ then $p_{0}\otimes I\in \Bb_{w,\tau_{q}}^{q}.$ Thus $$(p_{0}\otimes I)\left((\pi_{w}^{q}\boxtimes \tau_{q})\left(C_{w\cdot\lambda,\lambda}^{\lambda}\right)\right)(p_{0}\otimes I)\in \Bb_{w,\tau_{q}}^{q}$$ and is by Lemma~\ref{lem2} a non-zero constant multiple of $p_{0}\otimes \tau_{q}\left(C_{\lambda,\lambda}^{\lambda}\right).$ Since the functions $$\begin{array}{cccc}\tau_{q}\left(C_{\lambda,\lambda}^{\lambda}\right)\in C(\mathrm T),& \lambda\in P_{++}\end{array}$$
are generating $C(\mathrm T)$ as a $C^{*}$-algebra, it follows that \begin{equation}\label{p0}p_{0}\otimes C(\mathrm T)\subseteq \Bb_{w,\tau_{q}}^{q}.\end{equation}
By (Proposition 5.5 in~\cite{stokman}), the restriction of an irreducible $*$-representation of $C(\mathrm G)_{q}$ is still irreducible when restricted to 
$$
C(\mathrm{G/T})_{q}\overset{\ddef}{=}\{a\in C(\mathrm G)_{q}: (\iota\otimes \tau_{q})\circ \Delta_{q}(a)=a\otimes I \}.
$$
It follows that 
$$\K_{w}\otimes I\subseteq (\pi_{w}^{q}\boxtimes \tau_{q})(C(\mathrm{G/T})_{q})\subseteq (\pi_{w}^{q}\boxtimes \tau_{q})(C(\mathrm{G})_{q}),$$
and together with~\eqref{p0}, this gives~\eqref{subset}.

\end{proof}

\begin{proof}[Proof of Proposition~\ref{prop2}]
Recall the definition of $\mathrm T_{v}^{w}\subseteq \mathrm T.$ For $v\lhd w,$ we let $\chi_{v}^{w}$ be the composition $$\chi_{v}^{w}:C(\mathrm{G)_{q}}\overset{\tau_{q}}{\longrightarrow}C(\mathrm T)\longrightarrow C(\mathrm T_{v}^{w}),$$ where the second $*$-homomorphism is the restriction of $f\in C(\mathrm T)$ to $\mathrm T_{v}^{w}\subseteq \mathrm T.$ By definition, these $*$-representations are commutative and $q$-invariant. To prove the existence of the $*$-homomorphism, we use Theorem~\ref{ntthm} with $\mathrm Y=\{1\}$ (the set containing only the identity of $\mathrm T$). Thus for every $v\lhd w$ and $t\in \mathrm T_{v}^{w}$ we have that $\ker \pi_{w}^{q}\subseteq\ker \pi_{v}^{q}\boxtimes\chi_{t}$ and hence we can define a $*$-homomorphism 
$$\varphi_{v}^{w}:\Bb_{w}^{q}\longrightarrow \Bb_{v,\chi_{v}^{w}}^{q}$$
$$\begin{array}{ccc}\pi_{w}^{q}(x)\mapsto  (\pi_{v}^{q}\boxtimes \chi_{v}^{w})(x), & \text{for $x\in C(\mathrm{G)_{q}}.$}\end{array}$$
By Lemma~\ref{lem3}, we have $\pi_{w}^{q}(\Upsilon_{w})\neq 0$ and $\pi_{v}^{q}(\Upsilon_{w})=0$ and by Lemma~\ref{lem4} $$(\pi_{v}^{q}\boxtimes\chi_{v}^{w})(\Upsilon_{w})=\pi_{v}^{q}(\Upsilon_{w})\otimes I=0.$$  Hence the kernel of $\varphi_{v}^{w}$ is a non-trivial ideal of $\Bb_{w}^{q}$. By Lemma~\ref{lem5}, with $\chi=\epsilon_{q},$ we have $\K_{w}\subset \Bb_{w}^{q}.$ Thus any non-trivial ideal of $\Bb_{w}^{q}$ must contain $\K_{w}$ and therefore $\K_{w}\subseteq\ker \varphi_{v}^{w}.$ Now consider the $*$-homomorphism $$\sumrep{}\varphi_{v}^{w}:\Bb_{w}^{q}\longrightarrow \prod_{v\lhd w}\Bb_{v,\chi_{v}^{w}}^{q}.$$
By the definition of the $\varphi_{v}^{w}$'s, it follows that 
\begin{equation}\label{factors}
(\sumrep{}\varphi_{v}^{w})\circ \pi_{w}^{q}=\sumrep{}(\pi_{v}^{q}\boxtimes\chi_{v}^{w}).
\end{equation}
As $\K_{w}\subset \ker \sumrep{}\varphi_{v}^{w}$ we can factor $\sumrep{}\varphi_{v}^{w}=:\eta_{w}^{q}\circ p_{w}$ for a $*$-homomorphism
$$\eta_{w}^{q}:\Bb_{w}^{q}/\K_{w}\longrightarrow \sumrep{}(\pi_{v}^{q}\boxtimes \chi_{v}^{w})(C(\mathrm{G)_{q}}) .$$ From~\eqref{factors}, it follows that~\eqref{isomorph} holds. Clearly, by~\eqref{factors} $\eta_{w}^{q}$ is surjective. Hence we only need to show that the kernel of $\eta_{w}^{q}$ is trivial. By definition, this is the same as showing
\begin{equation}\label{compact1}
\ker \sumrep{}\varphi_{v}^{w}=\K_{w}.
\end{equation}
In order to prove this, we are going to show that, for any $x\in C(\mathrm{G)_{q}},$ if $(\pi_{v}^{q}\boxtimes \chi_{t})(x)=0$ for all $v\lhd w$ and $t\in \mathrm T_{v}^{w},$ then also $(\pi_{\sigma}^{q}\boxtimes \chi_{s})(x)=0$ for all $\sigma<w$ and $s\in \mathrm T_{\sigma}^{w}.$ To see this, notice that by~\eqref{mult}, there is a $\sigma\leq v\lhd w$ such that we have $s=s_{1}s_{2}$ for $s_{1}\in\mathrm T_{v}^{w}$ and $s_{2}\in \mathrm T_{\sigma}^{v}.$ As we have $(\pi_{v}^{q}\boxtimes \chi_{s_{1}})(x)=0$ we get from Theorem~\ref{ntthm} that also $(\pi_{r}^{q}\boxtimes \chi_{s})(x)=0.$ The equality~\eqref{compact1} now follows from Lemma~\ref{ntlem}.
\end{proof}
\begin{prop}\label{prop3}
Let $\mathcal S\subseteq W$ be a subset where all the elements have the same length i.e. $\elll(w)=\elll(v)$ for all $w,v\in \mathcal S.$ Moreover, assume that for each $v\in \mathcal S,$ we have a commutative $*$-homomorphism
$\chi_{v}:C(\mathrm{G)_{q}}\rightarrow C(\mathrm X_{v})$ such that $\chi_{v}(C(\mathrm{G)_{q}})=C(\mathrm X_{v}).$ 
Then 
\begin{equation}\label{elements}
\prod_{v\in \mathcal S}\K_{v}\otimes C(\mathrm X_{v})\subseteq \underset{v\in \mathcal S}{\sumr}(\pi_{v}\boxtimes \chi_{v})(C(\mathrm{G)_{q}})\subseteq \prod_{v\in \mathcal S}\Bb_{v,\chi_{v}}^{q}.
\end{equation}
\end{prop}
\begin{proof}
As all the elements of $\mathcal S$ have the same length, they must be mutually non-comparable in the partial ordering of $W.$ It follows from Lemma~\ref{lem3} and Lemma~\ref{lem4} that for $v\in \mathcal S,$ we have $(\pi_{v}^{q}\boxtimes\chi_{v})(\Upsilon_{v})=\pi_{v}^{q}(\Upsilon_{v})\otimes I\neq 0$ and $(\pi_{w}^{q}\boxtimes \chi_{w})(\Upsilon_{v})=0$ for any other $w\in \mathcal S.$ As $\pi_{v}^{q}(\Upsilon_{v})$ is a compact operator with dense range, it follows from Lemma~\ref{lem5} that 
$$\overline{(\pi_{v}^{q}\boxtimes\chi_{v})(\Upsilon_{v}(C(\mathrm{G)_{q}}))}=\K_{v}\otimes C(\mathrm X_{v})$$ 
$$\begin{array}{cccc}\overline{(\pi_{w}^{q}\boxtimes\chi_{w})(\Upsilon_{v}(C(\mathrm{G)_{q}}))}=\{0\},& \text{$w\in \mathcal S$ such that $w\neq v.$}\end{array}$$
This gives~\eqref{elements}.
\end{proof}
\subsection{\bf\textsc{Continuous Deformations}}
\begin{lem}\label{field}
There are invertible co-algebra maps
$$\begin{array}{ccccc}\theta^{q}:\mathbb{C}[\mathrm{G}]\longrightarrow \mathbb{C}[\mathrm{G]_{q}}, & q\in (0,1)\end{array}$$
such that for every $w\in W,$ every $q$-independent commutative $*$-representation $\chi^{q}:\mathbb{C}[\mathrm{G]_{q}}\to C(\mathrm X)$ and any fixed $f\in \mathbb{C}[G],$ the map
\begin{equation}\label{funcq}
q\in (0,1)\mapsto (\pi_{w}^{q}\boxtimes \chi^{q})(\theta^{q}(f))\in \Bb_{w,\chi^{q}}^{q}\subseteq \mathcal{B}(\Hh_{w})\otimes C(\mathrm X)
\end{equation}
is continuous.
\end{lem}
\begin{proof}
We will refer to the proof of Theorem~1.2 in~\cite{nt2}. We remark that our notation differs from theirs. Let $t_{ij}^{q},$ $i,j=1,2$ be the generators of $\mathbb{C}[\mathrm{SU_{2}]_{q}}$ for $q\in (0,1].$  It follows from the proof that there are invertible co-algebra maps 
$$
\begin{array}{cccc}
\kappa_{q}:\mathbb{C}[\mathrm{SU_{2}]}\longrightarrow \mathbb{C}[\mathrm{SU_{2}]_{q}}, & q\in(0,1], & \kappa_{1}=\id,
\end{array}
$$
such that $\kappa_{q}(t_{ij}^{1})=t_{ij}^{q}$ and for every $f\in \mathbb{C}[\mathrm{SU_{2}]},$ the image $\kappa_{q}(f)$ is a non-commutative polynomial in $t_{ij}^{q},$ $i,j=1,2,$ with coefficients continuous in $q.$ Moreover, there exists an invertible co-algebra map $\vartheta_{q}:\mathbb{C}[\mathrm{G]}\to  \mathbb{C}[\mathrm{G]_{q}}$, such that for every $i=1,\dots,n$, there exists a continuous family of invertible co-algebra morphisms $\gamma^{q}_{i}$ of $\mathbb{C}[\mathrm{G}]$ that makes the following diagram commute
\begin{equation}\label{diagram1}
\begin{xy}\xymatrixcolsep{3pc}\xymatrixrowsep{3pc}\xymatrix{
\mathbb{C}[\mathrm{SU_{2}]}\ar[d]_*{\kappa_{q_{i}}}& \mathbb{C}[\mathrm{G]}\ar[l]_*{\varsigma_{i}}\ar[r]^*{\gamma^{q}_{i}}  & \mathbb{C}[\mathrm{G]}\ar[dl]^*{\vartheta_{q}}\\
\mathbb{C}[\mathrm{SU_{2}]_{q_{i}}} & \mathbb{C}[\mathrm{G]_{q}}\ar[l]^*{\varsigma^{q}_{i}}&.
}\end{xy}
\end{equation}
This gives that \begin{equation}\label{kapvar}\kappa_{q_{i}}^{-1}\circ\varsigma_{i}^{q}\circ\vartheta_{q}=\varsigma_{i}\circ (\gamma_{i}^{q})^{-1}\end{equation} varies continuously on $q\in (0,1].$
The operators $C_{q}$ and $d_q,$ given by~\eqref{SC} varies continuously on $q\in (0,1).$ Hence, by~\eqref{SC1}, it follows that the functions $q\in(0,1)\mapsto\Pi_{q}\circ \kappa_{q}(t_{ij})=\Pi_{q}(t_{ij}^{q})\subseteq \mathcal{B}(\ell^{2}(\mathbb{Z}_{+})),$ for $i,j=1,2,$ are continuous. Thus, for any fixed $f\in \mathbb{C}[\mathrm{SU_{2}]},$ we have a continuous function
$$
q\in (0,1)\mapsto \Pi_{q}\circ \kappa_{q}(f)\in  \mathcal{B}(\ell^{2}(\mathbb{Z}_{+})).
$$
Composing this with~\eqref{kapvar}, we get that for every $f\in \mathbb{C}[\mathrm{G}],$ the function
$$
q\in (0,1)\mapsto \Pi_{q}\circ \varsigma_{i}^{q} \circ \vartheta_{q}( f)\in  \mathcal{B}(\ell^{2}(\mathbb{Z}_{+}))
$$
is continuous.
As  the maps in the diagram~\eqref{diagram1} are all (at least) co-algebra maps, so that $\Delta_{q}\circ\vartheta_{q}=(\vartheta_{q}\otimes \vartheta_{q})\circ \Delta,$ it follows for $f\in  \mathbb{C}[\mathrm{G]}$ and $w\in W$ with reduced presentation $w=s_{j_{1}}\cdots s_{j_{m}},$ that
$$
\pi_{w}^{q}(\vartheta_{q}(f))=(\pi_{j_{1}}^{q}\otimes \cdots\otimes \pi_{j_{m}}^{q})\circ \Delta^{(m)}_{q}(\vartheta_{q}(f))=
$$
$$
(\Pi_{q_{j_{1}}}\otimes \cdots\otimes \Pi_{q_{j_{m}}})\circ ((\varsigma_{j_{1}}^{q}\circ \vartheta_{q})\otimes\cdots \otimes(\varsigma_{j_{m}}^{q}\circ \vartheta_{q}))\circ \Delta^{(m)}(f)=
$$
$$
((\Pi_{q_{j_{1}}}\circ \kappa_{q_{j_{1}}}\circ \varsigma_{j_{1}})\otimes \cdots\otimes (\Pi_{q_{j_{m}}}\circ \kappa_{q_{j_{m}}}\circ \varsigma_{j_{m}}))\circ \Delta^{(m)}(f)
$$
and hence the function $q\in (0,1)\mapsto \pi_{w}^{q}(\vartheta_{q}(f))$ is continuous. Combining this with~\eqref{tauq}, it follows that $q\in (0,1)\mapsto\tau_{q}(\vartheta_{q}(f))\in C(\mathrm T)$ is also continuous. Thus 
$$
q\in(0,1)\mapsto(\pi_{w}^{q}\boxtimes \tau_{q})(\vartheta_{q}(f))=((\pi_{w}^{q}\circ \vartheta_{q})\otimes (\tau_{q}\circ \vartheta_{q}))\circ \Delta(f)\in \mathcal{B}(\Hh_{w})\otimes C(\mathrm T)
$$
is continuous. That~\eqref{funcq} holds for all $q$-independent maps follows by the factorization $\chi^{q}=\zeta\circ \tau_{q}.$ 
\end{proof}
Assume we have two subsets $\mathrm T_{1},\mathrm T_{2}\subseteq \mathrm T$ and $\mathrm T_{3}=\mathrm T_{1}\mathrm T_{2}$ (the point-wise multiplication). Let us denote by $\chi_{i},$ for $i=1,2,3$ the $*$-homomorphism $\mathbb{C}[\mathrm{G]_{q}}\to C(\mathrm T_{i}),$ $i=1,2,3,$ given by restriction of $\tau_{q}$ to $\mathrm T_{i}.$ It follows that we have an identification
\begin{equation}\label{multset}
\chi_{1}\boxtimes \chi_{2}\sim\chi_{3}
\end{equation}
in the sense that $\chi_{3}$ is the unique $*$-homomorphism with the property that, using the isomorphism $C(\mathrm T_{1})\otimes C(\mathrm T_{2})\cong C(\mathrm T_{1}\times T_{2}),$ we have
\begin{equation}\label{tensors2}
\begin{array}{cccc}
(\chi_{1}\boxtimes \chi_{2})(a)(t_{1},t_{2})=\chi_{3}(a)(t_{1}t_{2}), & t_{1}\in \mathrm T_{1}, & t_{2}\in\mathrm  T_{2}, & a\in \mathbb{C}[\mathrm{G]_{q}}.
\end{array}
\end{equation}
The multiplication map $m:\mathrm T_{1}\times \mathrm T_{2}\to \mathrm T_{3}$ gives an injective $*$-homomorphism 
$$
C(\mathrm T_{3})\overset{m^{*}}{\longrightarrow}C(\mathrm T_{1})\otimes C(\mathrm T_{2}) \cong C(\mathrm T_{1}\times T_{2})
$$
and it follows from~\eqref{tensors2} that $\chi_{1}\boxtimes\chi_{2}$ factors as
$$
\mathbb{C}[\mathrm{G]_{q}}\overset{\chi_{3}}{\longrightarrow} C(\mathrm T_{3})\overset{m^{*}}{\longrightarrow} C(\mathrm T_{1})\otimes C(\mathrm T_{2}).
$$
Furthermore, if for two subsets $\mathrm T_{1},\mathrm T_{2}\subseteq \mathrm T,$ we let $\mathrm T_{3}=\mathrm T_{1}\cup \mathrm T_{2}$ and denote by $\chi_{i},$ $i=1,2,3,$ the $*$-homomorphisms $\mathbb{C}[\mathrm{G]_{q}}\to C(\mathrm T_{i}),$ $i=1,2,3,$ then we have an identification
\begin{equation}\label{sumset}
\chi_{1}\oplus \chi_{2}\sim\chi_{3}
\end{equation}
via the injective $*$-homomorphism $C(\mathrm T_{1}\cup T_{2})\to C(\mathrm T_{1})\oplus C(\mathrm T_{2})$ determined by the two inclusions $\mathrm T_{i}\subseteq \mathrm T_{3}$ for $i=1,2.$ Thus $\chi_{3}$ satisfies
\begin{equation}\label{tensors}
\begin{array}{cccc}
\chi_{1}(a)=\chi_{3}(a)|_{\mathrm T_{1}}, &\chi_{2}(a)= \chi_{3}(a)|_{\mathrm T_{2}}, & a\in \mathbb{C}[\mathrm{G]_{q}}
\end{array}
\end{equation}
where $\chi_{3}(a)|_{\mathrm T_{i}},$ $i=1,2$ denotes the restriction of the function $\chi_{3}(a)\in C(\mathrm T_{3})$ to the subset $\mathrm T_{i}\subseteq \mathrm T_{3}.$
\\

For a path $v\overset{\gamma}{\leadsto}w,$ let us denote by $\chi_{\gamma}$ the commutative $q$-independent $*$-representation $\mathbb{C}[\mathrm{G]_{q}}\to C(\mathrm T_{\gamma}).$ If we have paths $v\overset{\gamma_{1}}{\leadsto}r\overset{\gamma_{2}}{\leadsto}w,$ then it follows from~\eqref{gamma} and~\eqref{multset} that if we have the composition of paths $\gamma=\gamma_{1}\circ \gamma_{2},$
then
\begin{equation}\label{cancel}
\chi_{\gamma_{1}}\boxtimes\chi_{\gamma_{2}}\sim\chi_{\gamma}.
\end{equation}
\section{\bf\textsc{The Main Result}}
\begin{thm}\label{main} 
\
\begin{enumerate}[(i)]
\item For all $q,s\in (0,1)$ and $w\in W,$ we have an inner $*$-automorphism $\Gamma_{w}^{s,q}:\mathcal{B}(\Hh_{w})\to \mathcal{B}(\Hh_{w})$ that restricts to a $*$-isomorphism $\Bb_{w}^{q}\to\Bb_{w}^{s},$ such that $\Gamma_{w}^{s,q}(\K_{w})=\K_{w}$ and we have $$\begin{array}{ccc}\Gamma_{w}^{s,t}\circ \Gamma_{w}^{t,q}=\Gamma_{w}^{s,q}, & \text{for all $s,t,q\in (0,1)$}\\ \Gamma_{w}^{q,q}=\id, & \text{for all $q\in(0,1)$}.\end{array}$$
Moreover, for all $q,s\in (0,1)$ and $w\in W,$ the following diagram commutes
\begin{equation}\label{arrows}
\begin{xy}\xymatrixcolsep{7pc}\xymatrixrowsep{3pc}\xymatrix{
B_{w}^{q}\ar[r]^*{\Gamma_{w}^{s,q}}\ar[d]_*{\eta_{w}^{q}\circ p_{w}}& B_{w}^{s}\ar[d]^*{\eta_{w}^{s}\circ p_{w}}\\
\prod_{v\lhd w}B_{v,\chi_{v}^{w}}^{q} \ar[r]_*{\prod_{v\lhd w}(\Gamma_{v}^{s,q}\otimes \iota)} & \prod_{v\lhd w}B_{v,\chi_{v}^{w}}^{s}
}\end{xy}
\end{equation}
where $\eta_{w}^{q}$ and $\chi_{v}^{w}$ are as in Proposition~\ref{prop2}. The $*$-isomorphisms $\Gamma_{w}^{s,q}$ are also continuous in the point-norm topology in the sense that, for fixed $q\in (0,1)$ and $y\in \Bb_{w}^{q},$ the function $s\in(0,1)\mapsto \Gamma^{s,q}_{w}(y)\in \mathcal{B}(\Hh_{w})$ is continuous.
\item If $\chi^{q}: \mathbb{C}[\mathrm{G]_{q}}\to C(\mathrm X),$ $q\in (0,1),$ are commutative $q$-independent $*$-homomorphisms, then the $*$-isomorphism $\Gamma_{w}^{s,q}\otimes \iota:\Bb_{w}^{q}\otimes C(\mathrm X)\to \Bb_{w}^{s}\otimes C(\mathrm X)$ restrics to a $*$-isomorphism $$\Gamma_{w}^{s,q}\otimes \iota:\Bb_{w,\chi}^{q}\longrightarrow\Bb_{w,\chi}^{s}.$$
\end{enumerate}
\end{thm}
\begin{proof}
We will prove $(i)$ and $(ii)$ simultaneously using induction on $k=\elll(w),$ starting at $k=0.$
\\

If $\elll(w)=0,$ then $w=e.$ As $\pi_{e}=\epsilon_{q},$ we have $\Bb_{e}=\mathbb{C}.$ By the definition of a $q$-independent commutative $*$-homomorphism $\chi^{q}=\zeta\circ \tau_{q},$ and hence $$\chi^{q}(C(\mathrm{G)_{q}})=\zeta\left(\tau_{q}(C(\mathrm{G)_{q}})\right)=\zeta(C(\mathrm T)).$$ So, $(i)$ and $(ii)$ hold in the case of $k=0$ with $\Gamma_{e}^{s,q}=\id_{\mathbb{C}}$. 
\\

Assume now that $(i)$ and $(ii)$ hold for all $v\in W$ of length $\elll(v)<k.$ By Proposition~\ref{prop2} we have a $*$-homomorphism
\begin{equation}\label{etas}
\partial_{1}^{q}:=\eta_{w}^{q}\circ p_{w}:\Bb_{w}^{q}\longrightarrow \Bb_{w}^{q}/\K_{w}\longrightarrow \prod_{v\lhd w}\Bb_{v,\chi_{v}^{w}}^{q}
\end{equation}
such that the image is isomorphic to $\Bb_{w}^{q}/\K_{w}=p_{w}(\Bb_{w}^{q})\subseteq \mathcal{Q}(\Hh_{w}).$ For $v\lhd w,$ we have
\begin{equation}\label{lhd} 
((\eta^{q}_{v}\circ p_{v})\otimes \iota)\circ (\pi^{q}_{v}\boxtimes \chi_{v}^{w})=\underset{\sigma\lhd v}{\prod}(\pi^{q}_{\sigma}\boxtimes (\chi_{\sigma}^{v}\boxtimes \chi_{v}^{w}))\sim \underset{\sigma\lhd v}{\prod}(\pi_{\sigma}\boxtimes \chi_{\gamma})
\end{equation}
where $\sigma\overset{\gamma}{\leadsto} w$ is the path $\sigma\lhd v\lhd w.$ Taking the product over all $v\lhd w$ we get from~\eqref{lhd} a $*$-homomorphism
\begin{equation}\label{haidt}
\partial_{2}^{q}:=\underset{v\lhd w}{\prod}(\eta_{v}^{q}\circ p_{v})\otimes \iota:\underset{v\lhd w}{\prod} \Bb_{v,\chi_{v}^{w}}^{q}\longrightarrow \overset{(2)}{\underset{\sigma\leadsto w}{\prod}}\Bb^{q}_{v,\chi_{\gamma}},
\end{equation}
where the product $\overset{(2)}{\underset{\sigma\leadsto w}{\prod}}$ is indexed over all $\sigma\lhd^{(2)} w$ and all possible paths $\sigma\overset{\gamma}{\leadsto} w.$ It follows from Lemma~\ref{lem5} that the kernel of $\partial_{2}^{q}$ is equal to $\underset{v\lhd w}{\prod}\K_{v}\otimes C(\mathrm T_{v}^{w}).$
If we iterate~\eqref{haidt}, then we get a sequence of $*$-homomorphisms
\begin{equation}\label{seq}
\Bb_{w}^{q}\overset{\partial_{1}^{q}}{\longrightarrow} \prod_{v\lhd w}\Bb_{v,\chi_{v}^{w}}^{q}\overset{\partial_{2}^{q}}{\longrightarrow}  \overset{(2)}{\underset{v\leadsto w}{\prod}}\Bb^{q}_{v,\chi_{\gamma}}\overset{\partial_{3}^{q}}{\longrightarrow}\dots\overset{\partial_{k-1}^{q}}{\longrightarrow} \overset{(k-1)}{\underset{v\leadsto w}{\prod}}\Bb_{v,\chi_{\gamma}}^{q}\overset{\partial_{k}^{q}}{\longrightarrow}  \underset{e\leadsto w}{\prod} C(\mathrm T_{\gamma}),
\end{equation}
where the product $ \overset{(i)}{\underset{v\leadsto w}{\prod}}$ ranges over all elements $v\in W$ such that $v\lhd^{(i)}w$ and over all possible paths $v\overset{\gamma}{\leadsto}w.$ In the last product, $e\in W$ is the identity element and we suppress the upper index $(k)$ as it is unnecessary in this case. In general, when we have a fixed element $v\in W$ such that $v\leq w,$ then $v\leadsto w$ denotes the set of all possible paths $v\overset{\gamma}{\leadsto} w.$ As an example, for $v\in W,$ we write $\underset{v\leadsto w}{\prod}\Bb_{v,\chi_{\gamma}}$ to mean that the product ranges over all possible paths $v\overset{\gamma}{\leadsto} w.$ Similarly, we write $\underset{v\lhd^{(i)} w}{\prod}$ to mean that the product is over all $v\in W$ such that $v\lhd^{(i)}w.$ Similar notations will also be used for direct sums, etc.
Clearly, by Lemma~\ref{lem5}, for every $i=1,\dots,k,$ we have
\begin{equation}\label{kalinka}\ker\partial_{i+1}^{q}= \overset{(i)}{\underset{v\leadsto w}{\prod}}\K_{v}\otimes C(\mathrm T_{\gamma}).\end{equation} Moreover, the commutative $C^{*}$-algebra $\underset{e\leadsto w}{\prod} C(\mathrm T_{\gamma})$ does not depend on $q.$ For $i=1,\dots,k,$ we let
\begin{equation*}
\partial_{i}^{q}\circ\dots\circ\partial_{1}^{q}=:\Psi_{i}^{q}:\Bb_{w}^{q}\longrightarrow   \overset{(i)}{\underset{v\leadsto w}{\prod}}\Bb^{q}_{v,\chi_{\gamma}}
\end{equation*}
\begin{equation*}
\partial_{k}^{q}\circ\dots\circ\partial_{1}^{q}=:\Psi_{k}^{q}:\Bb_{w}^{q}\longrightarrow   \underset{e\leadsto w}{\prod}C(\mathrm T_{\gamma})
\end{equation*}
be the composition of $*$-homomorphisms in~\eqref{seq}. By iteration of~\eqref{isomorph} and~\eqref{sumset}, we have for any $a\in\mathbb{C}[\mathrm{G]_{q}}$
\begin{equation}\label{comm}
\begin{array}{cc}
(\Psi_{i}^{q}\circ \pi_{w}^{q})(a)= \overset{(i)}{\underset{v\leadsto w}{\sumr}}(\pi_{v}^{q}\boxtimes\chi_{\gamma})(a), & i=1,\dots, k.
\end{array}
\end{equation}
By induction, for all $v\in W$ such that $l(v)<k,$ we have a $*$-isomorphism 
$$
\Gamma_{v}^{s,q}:\Bb_{v}^{q}\longrightarrow \Bb_{v}^{s}
$$
such that $\Gamma_{v}^{s,q}(\K_{v})=\K_{v}$ and the following diagram is commutative
$$
\begin{xy}\xymatrixcolsep{7pc}\xymatrixrowsep{3pc}\xymatrix{
B_{v}^{q}\ar[r]^*{\Gamma_{v}^{s,q}}\ar[d]_*{\eta_{v}^{q}\circ p_{v}}& B_{v}^{s}\ar[d]^*{\eta_{v}^{s}\circ p_{v}}\\
\underset{\sigma\lhd v}{\prod}B_{\sigma,\chi_{\sigma}^{v}}^{q} \ar[r]_*{\underset{\sigma\lhd v}{\prod}(\Gamma_{\sigma}^{s,q}\otimes \iota)} & \underset{\sigma\lhd w}{\prod} B_{\sigma,\chi_{\sigma}^{v}}^{s}
}\end{xy}
$$
It follows that for every $i=1,\dots,k-1,$ we have $*$-isomorphisms 
\begin{equation}\label{prodiso}
\overset{(i)}{\underset{v\leadsto w}{\prod}}(\Gamma_{v}^{s,q}\otimes \iota): \overset{(i)}{\underset{v\leadsto w}{\prod}}\Bb^{q}_{v,\chi_{\gamma}}\longrightarrow \overset{(i)}{\underset{v\leadsto w}{\prod}}\Bb^{s}_{v,\chi_{\gamma}}
\end{equation}
that maps $ \overset{(i)}{\underset{v\leadsto w}{\prod}}\K_{v}\otimes C(\mathrm T_{\gamma})$ into itself and such that the following diagrams are commutative
\begin{equation}\label{dia1}
\begin{xy}\xymatrixcolsep{5pc}\xymatrixrowsep{3pc}
\xymatrix{
\Bb_{w}^{q}  \ar[dr]_*{\Psi_{i+1}^{q}} \ar[r]^*{\Psi_{i}^{q}} &   \overset{(i)}{\underset{v\leadsto w}{\prod}}\Bb^{q}_{v,\chi_{\gamma}} \ar[d]_*{\partial_{i+1}^{q}}  \ar[r]^*{ \overset{(i)}{\underset{v\leadsto w}{\prod}}(\Gamma_{v}^{s,q}\otimes \iota)} &  \overset{(i)}{\underset{v\leadsto w}{\prod}}\Bb^{s}_{v,\chi_{\gamma}}  \ar[d]_*{\partial_{i+1}^{s}}  &\Bb_{w}^{s} \ar[l]_*{\Psi_{i}^{s}} \ar[dl]^*{\Psi_{i+1}^{s}}\\
& \overset{(i+1)}{\underset{v\leadsto w}{\prod}}\Bb^{q}_{v,\chi_{\gamma}}\ar[r]_*{   \overset{(i+1)}{\underset{v\leadsto w}{\prod}}(\Gamma_{v}^{s,q}\otimes \iota)}  &  \overset{(i+1)}{\underset{v\leadsto w}{\prod}}\Bb^{s}_{v,\chi_{\gamma}}&
}\end{xy}
\end{equation}
\begin{equation}\label{dia2}
\begin{xy}\xymatrixcolsep{5pc}\xymatrixrowsep{3pc}\xymatrix{
\Bb_{w}^{q}  \ar[r]^*{\Psi_{k-1}^{q}}  \ar[dr]_*{\Psi_{k}^{q}}& \overset{(k-1)}{\underset{v\leadsto w}{\prod}}\Bb^{q}_{v,\chi_{\gamma}}  \ar[d]_*{\partial_{k}^{q}} \ar[r]^*{ \overset{(k-1)}{\underset{v\leadsto w}{\prod}}(\Gamma_{v}^{s,q}\otimes \iota)}   & \overset{(k-1)}{\underset{v\leadsto w}{\prod}}\Bb^{s}_{v,\chi_{\gamma}}  \ar[d]_*{\partial_{k}^{s}} &\Bb_{w}^{s} \ar[l]_*{\Psi_{k-1}^{s}} \ar[dl]^*{\Psi_{k}^{s}}\\
& \underset{e\leadsto w}{\prod}C(\mathrm T_{\gamma}) \ar[r]_*{\id}  & \underset{e\leadsto w}{\prod}C(\mathrm T_{\gamma})  &.
}\end{xy}
\end{equation}
 The idea is now to show that $ \overset{(i)}{\underset{v\leadsto w}{\prod}}(\Gamma_{v}^{s,q}\otimes \iota)$ restricts to a $*$-isomorphism between $\Psi_{i}^{q}(\Bb_{w}^{q})$ and $\Psi_{i}^{s}(\Bb_{w}^{s})$ for $i=1,\dots,k.$ We prove this by 'climbing the ladder'~\eqref{dia1}, using induction on $i,$ starting at $i=k$  (i.e the case~\eqref{dia2}), and then we count down to $i=1$. 

The statement is clear at $k$, since by the $q$-independence of $\chi_{\gamma}$ and the fact that $\tau_{q}(C(\mathrm{G)_{q}})=C(\mathrm T)=\tau_{s}(C(\mathrm{G)_{s}}),$ we have $$\Psi_{k}^{q}(\Bb_{w}^{q})=(\underset{e\leadsto w}{\sumr}\chi_{\gamma})(C(\mathrm{G)_{q}})=(\underset{e\leadsto w}{\sumr}\chi_{\gamma})(C(\mathrm{G)_{s}})=\Psi_{k}^{s}(\Bb_{w}^{s}).$$
Assume now that the statement holds for $i+1.$ Consider $x\in \Psi_{i}^{q}(\Bb_{w}^{q}).$ Then $$\partial_{i+1}^{q}(x)\in \Psi_{i+1}^{q}(\Bb_{w}^{q}),$$ and hence by induction
$$
 \left(\overset{(i+1)}{\underset{v\leadsto w}{\prod}}(\Gamma_{v}^{s,q}\otimes \iota)\right)(\partial_{i+1}^{q}(x))\in \Psi_{k+1}^{s}(\Bb_{w}^{s}).
$$
But by the commutivity of the diagrams~\eqref{dia1}-\eqref{dia2}, this element is also equal to 
$$
\partial_{i+1}^{s}\left(\overset{(i)}{\underset{v\leadsto w}{\prod}}(\Gamma_{v}^{s,q}\otimes \iota) (x)\right)\in \Psi_{k+1}^{s}(\Bb_{w}^{s})
$$
from which it follows, by~\eqref{kalinka}, that
\begin{equation}\label{katyusha}
\overset{(i)}{\underset{v\leadsto w}{\prod}}(\Gamma_{v}^{s,q}\otimes \iota) (x)\in \Psi_{i}^{s}(\Bb_{w}^{s})+\overset{(i)}{\underset{v\leadsto w}{\prod}}\K_{v}\otimes C(\mathrm T_{\gamma})
\end{equation}
and thus
\begin{equation}\label{xelement}\begin{array}{cccc}\overset{(i)}{\underset{v\leadsto w}{\prod}}(\Gamma_{v}^{s,q}\otimes \iota) (x)=y+c,&y\in \Psi_{i}^{s}(\Bb_{w}^{s}),& c\in  \overset{(i)}{\underset{v\leadsto w}{\prod}}\K_{v}\otimes C(\mathrm T_{\gamma}).\end{array}\end{equation} We show that actually $c\in \Psi_{i}^{s}(\Bb_{w}^{s}).$ For a fixed $v\in W$ such that $v\lhd^{(i)}w,$ we can embed
\begin{equation}\label{embed}
\mathcal{B}(\Hh_{v})\otimes C(\mathrm T_{v}^{w})\subseteq \underset{v\leadsto w}{\prod}\mathcal{B}(\Hh_{v})\otimes C(\mathrm T_{\gamma})
\end{equation}
via the injective $*$-homomorphisms $C(\mathrm T_{v}^{w})\to \underset{v\leadsto w}{\prod}C(\mathrm T_{\gamma})$ coming from the inclusions $\mathrm T_{\gamma}\subseteq \mathrm T_{v}^{w}$ and, by the definition of $\mathrm T_{v}^{w}$, that $\mathrm T_{v}^{w}=\underset{v\leadsto w}{\cup} \mathrm T_{\gamma}.$ Thus the embedding~\eqref{embed} is on simple tensors given by $$\begin{array}{ccc}x\otimes f\mapsto\underset{v\leadsto w}{\prod}(x\otimes f|_{\mathrm T_{\gamma}}), & x\in \mathcal{B}(\Hh_{v}), & f\in C(\mathrm T_{v}^{w}),\end{array}$$
where $f|_{\mathrm T_{\gamma}}$ denoted the restriction of $f\in C(\mathrm T_{v}^{w})$ to the subset $\mathrm T_{\gamma}\subseteq \mathrm T_{v}^{w}.$ Moreover, we have the embedding
\begin{equation}\label{embedcompact}
\K_{v}\otimes C(\mathrm T_{v}^{w})\subseteq\underset{v\leadsto w}{\prod}\mathcal{B}(\Hh_{v})\otimes C(\mathrm T_{\gamma})
\end{equation}
coming from~\eqref{embed}. Using this embedding, we clearly have, for fixed $v\lhd^{(i)}w,$ that  
\begin{equation}\label{embedrange}
\underset{v\leadsto w}{\sumr}(\pi_{v}^{q}\boxtimes \chi_{\gamma}):\mathbb{C}[\mathrm{G]_{q}}\longrightarrow \mathcal{B}(\Hh_{v})\otimes C(\mathrm T_{v}^{w})\subseteq \underset{v\leadsto w}{\prod}\mathcal{B}(\Hh_{v})\otimes C(\mathrm T_{\gamma})
\end{equation}
and that, as in~\eqref{sumset}, we can identify $\underset{v\leadsto w}{\sumr}(\pi_{v}^{q}\boxtimes\chi_{\gamma})\sim\pi_{v}^{q}\boxtimes \chi_{v}^{w}.$ It then follows from Lemma~\ref{lem5}, that under the embeddings~\eqref{embed} and~\eqref{embedcompact} we have
\begin{equation}\label{subsets}
\K_{v}\otimes C(\mathrm T_{v}^{w})\subseteq \overline{\underset{v\leadsto w}{\sumr}(\pi_{v}^{q}\boxtimes \chi_{\gamma})(\mathbb{C}[\mathrm{G]_{q}})}\subseteq \mathcal{B}(\Hh_{v})\otimes C(\mathrm T_{v}^{w})\subseteq \underset{v\leadsto w}{\prod}\mathcal{B}(\Hh_{v})\otimes C(\mathrm T_{\gamma}).
\end{equation}
Moreover, note that the left-hand sides~\eqref{embed} and~\eqref{embedcompact} are clearly invariant under the homomorphism $\underset{v\leadsto w}{\prod}\Gamma^{s,q}_{v}\otimes \iota.$  
We can now use Proposition~\ref{prop3} and~\eqref{comm} to see that if we take the product of~\eqref{subsets}, ranging over all $v\lhd^{(i)}w,$ then
\begin{equation}
\underset{v\lhd^{(i)}w}{\prod}\K_{v}\otimes C(\mathrm T_{v}^{w})\subseteq \Psi_{i}^{q}(\Bb^{q}_{w})\subseteq \underset{v\lhd^{(i)}w}{\prod}\mathcal{B}(\Hh_{v})\otimes C(\mathrm T_{v}^{w})\subseteq \overset{(i)}{\underset{v\leadsto w}{\prod}}\mathcal{B}(\Hh_{v})\otimes C(\mathrm T_{\gamma}).
\end{equation}
As $\overset{(i)}{\underset{v\leadsto w}{\prod}}(\Gamma_{v}^{s,q}\otimes \iota)$ clearly fixes the two sub-algebras on either side of $\Psi_{i}^{q}(\Bb_{w}^{q}),$ it follows from~\eqref{xelement} that
$$
\overset{(i)}{\underset{v\leadsto w}{\prod}}(\Gamma_{v}^{s,q}\otimes \iota) (x)-y=c\in \underset{v\lhd^{(i)}w}{\prod}\K_{v}\otimes C(\mathrm T_{v}^{w})\subseteq \Psi_{i}^{s}(\Bb^{s}_{w}).
$$
From this, it follows that 
\begin{equation}\label{equity}\begin{array}{ccc}\overset{(i)}{\underset{v\leadsto w}{\prod}}(\Gamma_{v}^{s,q}\otimes \iota)( \Psi_{i}^{q}(\Bb_{w}^{q}))\subseteq  \Psi_{i}^{s}(\Bb_{w}^{s}), & q,s\in (0,1).\end{array}\end{equation}
But as $$\begin{array}{cccc} \overset{(i)}{\underset{v\leadsto w}{\prod}}(\Gamma_{v}^{s,q}\otimes \iota)\circ \overset{(i)}{\underset{v\leadsto w}{\prod}}(\Gamma_{v}^{q,s}\otimes \iota)=\id, & q,s\in (0,1)\end{array}$$ we must have equality in~\eqref{equity}.
\\

Thus, we have an isomorphism $$\begin{array}{ccc}\Bb_{w}^{q}/\K_{w}\cong\Bb_{w}^{s}/\K_{w},  & q,s\in (0,1)\end{array}$$  via the $*$-isomorphism 
\begin{equation}\label{tiso}\mathcal{L}^{s,q}:=(\eta_{w}^{s})^{-1}\circ \left(\underset{v\lhd w}{\prod}(\Gamma_{v}^{s,q}\otimes \iota)\right)\circ \eta_{w}^{q}.\end{equation}
However, to be able to use \fact~to conclude that the $C^{*}$-algebras $\Bb_{w}^{q},$ $q\in (0,1)$ are all isomorphic, we must also show that $\mathcal{L}^{s,q}$ are continuous in the point-norm topology, i.e. that for a fixed $q\in (0,1)$ and an element $y\in \Bb_{w}^{q}/\K_{w},$ we have a continuous function
\begin{equation}\label{continuous}
\begin{array}{ccc}
s\in(0,1)\to \mathcal{Q}(\Hh_{w}), & s\mapsto \mathcal{L}^{s,q}(y).
\end{array}
\end{equation}
By a classical approximation argument, it is enough to prove this for the dense $*$-subalgebra $(p_{w}\circ\pi_{w}^{q})(\mathbb{C}[\mathrm{G]_{q}}).$
By Lemma~\ref{field} we have invertible coalgebra morphisms $\theta^{q}:\mathbb{C}[\mathrm{G}]\to \mathbb{C}[\mathrm{G]_{q}}$ such that for fixed $f\in \mathbb{C}[\mathrm{G}],$ the function $q\in (0,1)\mapsto \pi_{w}^{q}(\theta^{q}(f))\in \mathcal{B}(\Hh_{w})$ is continuous. Thus the function $$q\in(0,1)\mapsto (p_{w}\circ \pi_{w}^{q})(\theta^{q}(f))\in \mathcal{B}(\Hh_{w})/\K_{w}=\mathcal{Q}(\Hh_{w})$$ is also continuous.  Let us write $$F^{q}:=(p_{w}\circ\pi_{w}^{q})(\theta^{q}(f))\in \mathcal{Q}(\Hh_{w}).$$ By induction, the function 
$$
s\in(0,1)\mapsto \left(\underset{v\lhd w}{\prod}(\Gamma_{v}^{s,q}\otimes \iota)\right)(\eta_{w}^{q}(F^{q}))
$$
is continuous and $\left(\underset{v\lhd w}{\prod}(\Gamma_{v}^{q,q}\otimes \iota)\right)(\eta_{w}^{q}(F^{q}))=\eta_{w}^{q}(F^{q})$. Notice that by the definition of $\eta^{q}_{w}$ and~\eqref{factors}, we have $$\eta_{w}^{q}(F^{q})=\sumrep{}(\pi_{v}^{q}\boxtimes \chi_{v}^{w})(\theta^{q}(f))$$ and thus by Lemma~\ref{field}, the function
$$
q\in (0,1)\mapsto \eta_{w}^{q}(F^{q})\in \prod_{v\lhd w}\mathcal{B}(\Hh_{v})\otimes C(\mathrm T_{v}^{w})
$$
is continuous. It follows that for all $\epsilon>0$ we have
$$
\|\left(\underset{v\lhd w}{\prod}(\Gamma_{v}^{s,q}\otimes \iota)\right)(\eta_{w}^{q}(F^{q}))- (\eta_{w}^{s}(F^{s})) \|\leq 
$$
$$
\|\left(\underset{v\lhd w}{\prod}(\Gamma_{v}^{s,q}\otimes \iota)\right)(\eta_{w}^{q}(F^{q}))-\eta_{w}^{q}(F^{q}) \|+\| \eta_{w}^{q}(F^{q})- \eta_{w}^{s}(F^{s}) \|< \epsilon
$$
for $|s-q|<\delta_{1},$ if $\delta_{1}>0$ is made small enough. If we apply the $*$-isomorphism $(\eta_{w}^{s})^{-1},$ we get
\begin{equation}\label{smaller}
\begin{array}{ccc}
\|\mathcal{L}^{s,q}\left(F^{q}\right)-F^{s}\|<\epsilon, & \text{for $|s-q|<\delta_{1}$}
\end{array}
\end{equation}
and thus it follows that there is a $0<\delta \leq \delta_{1},$ such that
$$
\|\mathcal{L}^{s,q}\left(F^{q}\right)-F^{q}\|\leq
$$
$$
\|\mathcal{L}^{s,q}\left(F^{q}\right)-F^{s}\|+\|F^{s}-F^{q}\|<2\epsilon
$$
when $|s-q|<\delta.$ We can now apply \fact~to get an inner $*$-automorphism $\Gamma_{w}^{s,q}:\mathcal{B}(\Hh_{w})\to\mathcal{B}(\Hh_{w})$ that restricts to a $*$-isomorphism $\Bb_{w}^{q}\to \Bb_{w}^{s},$ that is continuous in the point-set topology. That the diagram~\eqref{arrows} commutes follows from the commutivity of~\eqref{com} and the way $\mathcal{L}^{s,q}$ was defined. Clearly, the compact operators are invariant under $\Gamma_{w}^{s,q}.$
\\

The case $(ii).$ We prove it first for $\tau_{q}:\mathbb{C}[\mathrm{G]_{q}}\to C(\mathrm T)$ (see~\eqref{tauq}). We combine the inclusion $\Bb_{w,\tau_{q}}^{q}\hookrightarrow \Bb_{w}^{q}\otimes C(\mathrm T)$ with the sequence~\eqref{seq} by tensoring all the components with $C(\mathrm T)$ in the following way
$$
\Bb^{q}_{w,\tau_{q}}\hookrightarrow \Bb^{q}_{w}\otimes C(\mathrm T) \overset{\partial_{1}\otimes\iota}{\longrightarrow} \left(\prod_{v\lhd w}\Bb^{q}_{v,\chi_{v}^{w}}\right)\otimes C(\mathrm T)\overset{\partial_{2}\otimes\iota}{\longrightarrow} \left(\overset{(2)}{\underset{v\leadsto w}{\prod}}\Bb^{q}_{v,\chi_{\gamma}}\right)\otimes C(\mathrm T)\overset{\partial_{3}\otimes\iota}{\longrightarrow}\dots
$$
\begin{equation}\label{tensorseq}
\dots\overset{\partial_{k-1}\otimes\iota}{\longrightarrow} \left(\overset{(k-1)}{\underset{v\leadsto w}{\prod}}\Bb^{q}_{v,\chi_{\gamma}}\right)\otimes C(\mathrm T)\overset{\partial_{k}\otimes\iota}{\longrightarrow} \left(\underset{e\leadsto w}{\prod}C(\mathrm T_{\gamma})\right)\otimes C(\mathrm T).
\end{equation}
If we define $\Psi_{i}^{q}$ as before, then similar to~\eqref{comm}, we have
\begin{equation}\label{comm2}
(\Psi_{i}^{q}\otimes \iota)\circ (\pi_{w}^{q}\boxtimes \tau_{q})= \left(\underset{v\leadsto w}{\overset{(i)}{\sumr}}(\pi_{v}^{q}\boxtimes\chi_{\gamma})\right)\boxtimes \tau_{q}.
\end{equation}
We can proceed exactly as before, using the commutative diagrams~\eqref{dia1} and~\eqref{dia2} (now tensored by $C(\mathrm T)$), by induction on $i=1,\dots , k,$ starting at $k.$ Clearly the images in $\left(\underset{e\leadsto w}{\prod}C(\mathrm T_{\gamma})\right)\otimes C(\mathrm T)$ are the same, as the commutative $*$-representation is $q$-independent. Assuming that $$\left(\overset{(i+1)}{\underset{v\leadsto w}{\prod}}(\Gamma_{v}^{s,q}\otimes \iota)\right)\otimes \iota:\left(\Psi^{q}_{i+1}\otimes \iota\right)(\Bb_{w,\tau_{q}}^{q})\longrightarrow\left(\Psi^{s}_{i+1}\otimes \iota\right)(\Bb_{w,\tau_{s}}^{s})$$
is an $*$-isomorphism gives for $x\in \left(\Psi^{q}_{i}\otimes \iota\right)(\Bb_{w,\tau_{q}}^{q}),$ that
\begin{equation}\label{korobeiniki}
\left(\left(\overset{(i)}{\underset{v\leadsto w}{\prod}}(\Gamma_{v}^{s,q}\otimes \iota)\right)\otimes \iota\right)(x)\in \left(\Psi^{s}_{i}\otimes \iota\right)(\Bb_{w,\tau_{s}}^{s})+\left(\overset{(i)}{\underset{v\leadsto w}{\prod}}\K_{v}\otimes C(\mathrm T_{\gamma})\right)\otimes C(\mathrm T).
\end{equation}

The rest of the argument follows in a similar fashion as for $\Bb_{w}^{q}:$ we find an embedding 

\begin{equation}\label{embed2}
\underset{v\lhd^{(i)}w}{\prod}\mathcal{B}(\Hh_{v})\otimes C(\mathrm T)\subseteq \left(\overset{(i)}{\underset{v\leadsto w}{\prod}}\mathcal{B}(\Hh_{v})\otimes C(\mathrm T_{\gamma})\right)\otimes C(\mathrm T)\end{equation}
such that, 
\begin{enumerate}[(i)]
\item this subalgebra is invariant with respect to the map $\left(\overset{(i)}{\underset{v\leadsto w}{\prod}}(\Gamma_{v}^{s,q}\otimes \iota)\right)\otimes\iota,$
 \item we have
\begin{equation}\label{intersec}
\underset{v\lhd^{(i)}w}{\prod}\K_{v}\otimes C(\mathrm T)=\underset{v\lhd^{(i)}w}{\prod}\mathcal{B}(\Hh_{v})\otimes C(\mathrm T)\bigcap \left(\overset{(i)}{\underset{v\leadsto w}{\prod}}\K_{v}\otimes C(\mathrm T_{\gamma})\right)\otimes C(\mathrm T)
\end{equation}
\item and the following inclusions holds
\begin{equation}\label{subsets2}
\underset{v\lhd^{(i)}w}{\prod}\K_{v}\otimes C(\mathrm T)\subseteq\left(\Psi^{s}_{i}\otimes \iota\right)(\Bb_{w,\tau_{s}}^{s}) \subseteq \underset{v\lhd^{(i)}w}{\prod}\mathcal{B}(\Hh_{v})\otimes C(\mathrm T).
\end{equation}
\end{enumerate}
Clearly, this implies $$\left(\left(\overset{(i)}{\underset{v\leadsto w}{\prod}}(\Gamma_{v}^{s,q}\otimes \iota)\right)\otimes \iota\right)(x)\in \left(\Psi^{s}_{i}\otimes \iota\right)(\Bb_{w,\tau_{s}}^{s}).$$
To do this, we use the natural isomorphism
$$
\left(\overset{(i)}{\underset{v\leadsto w}{\prod}}\mathcal{B}(\Hh_{v})\otimes C(\mathrm T_{\gamma})\right)\otimes C(\mathrm T)\cong \overset{(i)}{\underset{v\leadsto w}{\prod}}\mathcal{B}(\Hh_{v})\otimes C(\mathrm T_{\gamma})\otimes C(\mathrm T).
$$
Notice that for every $v\overset{\gamma}{\leadsto} w,$ we have an embedding $C(\mathrm T)\subseteq C(\mathrm T_{\gamma})\otimes C(\mathrm T)$ determined by $C(\mathrm T_{\gamma})\otimes C(\mathrm T)\cong C(\mathrm T_{\gamma}\times T)$ and the multiplication map $\mathrm T_{\gamma}\times \mathrm T\to \mathrm T.$ As $\tau_{q}:\mathbb{C}[\mathrm{G]_{q}}\to \mathbb{C}[\mathrm T]$ is a morphism of Hopf $*$-algebras (hence compatible with the multiplication in $\mathrm T$), it follows that we have a $*$-homomorphism
$$
\pi_{v}^{q}\boxtimes \chi_{\gamma}\boxtimes \tau_{q}:C(\mathrm{G)_{q}}\to \mathcal{B}(\Hh_{v})\otimes C(\mathrm T)\subseteq \mathcal{B}(\Hh_{v})\otimes C(\mathrm T_{\gamma})\otimes C(\mathrm T).
$$
We can then, for fixed $v\lhd^{(i)}w,$ embed diagonally 
$$
\mathcal{B}(\Hh_{v})\otimes C(\mathrm T)\subseteq \underset{v\leadsto w}{\prod}\mathcal{B}(\Hh_{v})\otimes C(\mathrm T_{\gamma})\otimes C(\mathrm T).
$$
By taking the product over all $v\lhd^{(i)}w,$ we get an embedding~\eqref{embed2} such that~\eqref{subsets2} holds (the first inclusion follows from Proposition~\eqref{prop3}). Clearly, we also have~\eqref{intersec}. Thus, it follows that $(\Gamma^{s,q}_{w}\otimes \iota)(\Bb_{w,\tau_{q}}^{q})=\Bb_{w,\tau_{q}}^{s},$ since this is the case $i=1.$
\\

This implies the general case: let $\chi^{q}:C(\mathrm{G)_{q}}\to C(\mathrm X)$ be commutative $q$-independent $*$-homomorphisms and $\zeta:C(\mathrm T)\to C(\mathrm X)$ the $*$-homomorphism such that $\chi^{q}=\zeta\circ \tau_{q}.$ Then $\iota\otimes \zeta$ is a surjective $*$-homomorphism $\Bb^{q}_{w,\tau_{q}}\to \Bb^{q}_{v,\chi^{q}}.$
As $\Gamma_{w}^{s,q}\otimes \iota$ is a $*$-isomorphism $\Bb_{w,\tau_{q}}^{q}\to \Bb_{w,\tau_{s}}^{s},$ we have
$$
(\Gamma_{w}^{s,q}\otimes \iota)(\Bb_{w,\chi^{q}}^{q})=(\Gamma_{w}^{s,q}\otimes \iota)\circ (\iota\otimes \zeta) (\Bb_{w,\tau_{q}}^{q})=(\iota\otimes \zeta)\circ (\Gamma_{w}^{s,q}\otimes \iota) (\Bb_{w,\tau_{q}}^{q})=
$$
$$
=(\iota\otimes \zeta)(\Bb_{w,\tau_{s}}^{s})=\Bb_{w,\chi^{s}}^{s}.
$$
\end{proof}
\begin{cor}
The universal enveloping $C^{*}$-algebras of $\mathbb{C}[\mathrm{G]_{q}}$ are isomorphic for all $q\in (0,1).$ These isomorphisms are equivariant with respect to the right-action of $\mathrm T.$
\end{cor}
\begin{proof}
If $\omega\in W$ is the unique element of longest length in the Weyl group and $\tau_{q}:C(\mathrm{G)_{q}}\to C(\mathrm T)$ the commutative $q$-independent $*$-homomorphism coming from the embedding of the maximal torus $\mathrm T\subseteq \mathrm G,$ then it follows from Theorem~\ref{ntthm} that any irreducible $*$-representation of $C(\mathrm{G)_{q}}$ must factor through $\pi_{\omega}^{q}\boxtimes \tau_{q}.$ Thus $\Bb_{\omega,\tau_{q}}^{q}\cong C(\mathrm{G)_{q}}.$ The q-independence follows from Theorem~\ref{main}(ii). For $x\in C(\mathrm G)_{q}$ and $t\in \mathrm T$ we have 
$$
(\pi^{q}_{\omega}\boxtimes \tau_{q})(R_{t}(x))=
$$
$$
=(\iota\otimes\iota\otimes \ev_{t})\circ(\iota\otimes \Delta_{\mathrm T})((\pi_{\omega}^{q}\boxtimes\tau_{q})(x))
$$
and thus equivariance with respect to the right-action follows as the isomorphism $\Bb_{\omega,\tau_{q}}^{q}\to \Bb_{\omega,\tau_{s}}^{s}$ is of the form $\Gamma^{s,q}_{\omega}\otimes \iota.$

\end{proof}

\end{document}